\documentclass[11pt]{article}
\usepackage{preamble}
\usepackage{times}

\usepackage[backend=biber, maxbibnames=99]{biblatex}
\AtEveryBibitem{
    \clearfield{url}
    \clearfield{issn}
    \clearfield{urlyear}
    \clearfield{urlmonth}
    \clearfield{month}
    \clearfield{language}
}

\newcommand{\footremember}[2]{
   \footnote{#2}
    \newcounter{#1}
    \setcounter{#1}{\value{footnote}}
}

\begin{document}

\title{Transfer Learning under Covariate Shift: Local $k$-Nearest Neighbours Regression with Heavy-Tailed Design}

\author{
    Petr Zamolodtchikov\footremember{alley}{Department of Applied Mathematics, University of Twente, 7522 NB Enschede, The Netherlands}\footremember{trailer}{The research has been supported by the NWO Vidi grant
VI.Vidi.192.021.}
    \ and Hanyuan Hang 
}
\date{\today}

\maketitle

\begin{abstract}
  Covariate shift is a common transfer learning scenario where the marginal distributions of input variables vary between source and target data while the conditional distribution of the output variable remains consistent. The existing notions describing differences between marginal distributions face limitations in handling scenarios with unbounded support, particularly when the target distribution has a heavier tail. To overcome these challenges, we introduce a new concept called \textit{density ratio exponent} to quantify the relative decay rates of marginal distributions' tails under covariate shift. Furthermore, we propose the \textit{local k-nearest neighbour regressor} for transfer learning, which adapts the number of nearest neighbours based on the marginal likelihood of each test sample. From a theoretical perspective, convergence rates with and without supervision information on the target domain are established. Those rates indicate that our estimator achieves faster convergence rates when the density ratio exponent satisfies certain conditions, highlighting the benefits of using density estimation for determining different numbers of nearest neighbours for each test sample. Our contributions enhance the understanding and applicability of transfer learning under covariate shift, especially in scenarios with unbounded support and heavy-tailed distributions.
\end{abstract}

\section{Introduction}
\label{sec.intro}

Transfer learning aims to leverage the knowledge gained from the \textit{source} distribution $\P$ to improve the performance on the \textit{target} distribution $\Q$, especially when limited data is available from the target distribution.
In contrast with traditional machine learning approaches, where models are trained and evaluated on the same distribution, transfer learning allows models to generalize knowledge across source and target distributions, potentially improving performance on the target distribution with limited data.
Transfer learning has been successfully applied in various domains, including computer vision \cite{li2020transfer,wang2022transfer}, natural language processing \cite{ruder2019transfer}, healthcare \cite{ali2021enhanced,ebbehoj2022transfer}, speech recognition \cite{shivakumar2020transfer} and so on. 
By adapting the knowledge of the source distribution to the target distribution, transfer learning helps to improve the performance of the learning algorithms in these fields.

In this article, we focus on one of the most common settings in transfer learning, i.e., \emph{covariate shift} \cite{shimodaira2000improving}. 
Under this setting, the marginal distributions of the input variables may differ from the source to target data, whereas the conditional distribution of the output variable, given the input variables, remains the same across the source and the target distribution. Covariate shift indicates that there has been a change in the distribution of input features across domains. This change may be caused by factors such as different periods, different devices, and diverse environments, leading to variations in the data distribution. For example, consider a computer vision scenario where source data consists of images captured during the day, while target data comprises images taken at night. Due to differences in lighting conditions and other factors, a covariate shift exists between the two sets, potentially leading to a decline in the performance of a model trained on nighttime images when applied to daytime images. In addition to this example, covariate shift is also conceivable in a variety of real-world fields, including semantic segmentation \cite{richter2016playing,cordts2016cityscapes}, speech recognition \cite{sainath15b_interspeech}, sentiment analysis \cite{blitzer2006domain}, and reinforcement learning \cite{cobbe2019quantifying}.

In the theoretical studies of algorithms under covariate shift, 
various notions have been proposed to describe differences in feature probabilities between the source and target domains. 
However, these notions typically focus on scenarios with bounded support and do not account for unbounded support situations.
For instance, \cite{kpotufe2021marginal} introduces the concept called transfer-exponent that parameterizes ball-mass ratios $\Q\{\B(x,r)\}/\P\{\B(x,r)\}$ to encode the singularity of $\Q$ with respect to $\P$. 
However, if $\Q$ has heavier tail than $\P$, the transfer-exponent $\gamma$ is infinite. Consequently, according to the convergence rates established in \cite{kpotufe2021marginal}, no matter how much data is provided, it cannot improve the convergence rates of classifiers in the target domain, implying that source domain data does not assist in predicting the target domain. 
More recently, \cite{pathak2022new} introduces a new similarity measure between two probability measures. However, the similarity measure becomes infinite when $\Q$ has heavier tails than $\P$.
In this case, the theoretical results established by \cite{pathak2022new} are no longer valid.
Notice that under the covariate shift assumption, where the source domain and the target domain share the same conditional distribution of labels given features, the relationship learned between features and labels on $\P$ may still be effective on $\Q$, indicating an improvement of predictions on the target domain with the usage of the data from $\P$. 
However, existing notions and the corresponding theoretical results cannot prove this point, indicating that in some cases, the exponent cannot effectively describe the relationship between $\P$ and $\Q$ with unbounded support, especially when $\Q$ has heavier tails than $\P$. This reflects the limitation of these notions.

To address these challenges, we first introduce a new concept called \textit{density ratio exponent} to describe the relationship between the marginal distributions of $\P$ and $\Q$ under covariate shift, which can quantify the relative decay rates of the tails of the source and target marginal distributions. Furthermore, based on the information from the marginal distributions of $\P$ and $\Q$,
we propose a \textit{local} $k$-nearest neighbour regressor ($k$-NN) for transfer learning under covariate shift, which assigns different numbers of nearest neighbours to test samples. Compared to the standard $k$-NN regressor, the local choice of $k$ depends on the relevance of test instances to the source distribution. 
Specifically, if a test instance belongs to a high-probability region of the source distribution, more neighbours of this instance are used for prediction. By contrast, if a test instance is unlikely to appear in the source domain, only a few neighbours are used for prediction.

The contributions of this article can be summarised as follows.

\textit{$(i)$} 
We introduce the concept of the \textit{density ratio exponent} to characterise the relationship between the marginal density functions of the source and target distributions. In comparison to existing notions under the covariate shift assumption, the density ratio exponent can quantify the relative decay rates between the tails of $\P$ and $\Q$.

\textit{$(ii)$}
We propose a \textit{local} $k$-NN for transfer learning under covariate shift. From a theoretical perspective, convergence rates are established for both the supervised and unsupervised cases when $k$ is allowed to vary for each test sample based on its marginal distribution.

\textit{$(iii)$}
Our theoretical results indicate that, compared to the standard $k$-NN regressor, our estimator can achieve faster convergence rates when the density ratio exponent satisfies certain conditions. This demonstrates the advantages of utilising density estimation to determine varying numbers of nearest neighbours for each test sample from the perspective of convergence rates.\\

\noindent \textbf{Paper organization.} We introduce the statistical model, our working assumptions, and the local $k$-NN regressor in Section \ref{sec.prelims}. We then state and discuss our main results in Section \ref{sec.results}. An outline of the proof for the rates of our two-sample local $k$-NN is given in Section \ref{sec.proof.outline} before concluding and drawing future research directions in Section \ref{sec.conclusion}. All the proofs are deferred to the appendix.

\section{Preliminaries}\label{sec.prelims}

\noindent\textbf{Notations}
We use the notation $a_n \lesssim b_n$ to denote that there exists a positive constant $c$ and an integer $N$ such that $a_n \leq c b_n$, for all $n \geq N$. $a_n \gtrsim b_n$ is defined in a similar way, and we denote $a_n \asymp b_n$ whenever $a_n \lesssim b_n$ as well as $a_n \gtrsim b_n.$
For any $x \in \RR^d$ and $r > 0$, we denote $\B(x,r) := \{ x' \in \mathbb{R}^d : \|x' - x\| \leq r \}$ as the closed ball centered in $x$ with radius $r$ in the Euclidean norm denoted by $\|\cdot\|.$ The infinite norm of a function $f$ is denoted by $\|f\|_{\infty}.$ Given a sample $(X_1, Y_1), \dots, (X_n, Y_n),$ and a point $x \in \RR^d,$ we denote by $X_i(x)$ the $i$-th closest neighbour of $x$ among $\{X_1, \dots, X_n\}.$ That is, for all $x \in \RR^d,$ we have $\{X_1(x), \dots, X_n(x)\} = \{X_1, \dots, X_n\},$ and $\|x - X_1(x)\| \leq \dots \leq \|x - X_n(x)\|.$
We denote by $Y_i(x)$ the label of $X_i(x),$ and by $R_i(x) := \|x - X_i(x)\|,$ the $i$-th nearest neighbour distance.
For $n \in \NN$ and $x_1, \dots, x_n \in \RR^d,$ we denote $x^n := (x_1, \dots, x_n).$ Similarly, for a collection of $\RR^d$-valued random variables $X_1, \dots, X_n,$ we denote $X^n := (X_1, \dots, X_n).$ Since we consider distributions $\P$ that admit a density $p$ with respect to the Lebesgue measure on $\RR^d,$ we denote $\supp(\P) := \supp(p) := \{x \in \RR^d: p(x) > 0\}.$ Let $\mM(\ol p)$ be the class of probability distributions $\P$ on $\big(\RR^d, \mB(\RR^d)\big)$ that are absolutely continuous with respect to the Lebesgue measure on $(\RR^d, \mB(\RR^d)),$ and with density $p$ such that $\|p\|_\infty \leq \ol p < \infty.$\\

\noindent\textbf{Nonparametric regression model under covariate shift.} We consider the standard nonparametric regression model. Let $(X, Y) \sim \Q_{\sX, \sY}$ be a couple of random variables with values in $\RR^d \times \RR.$ We assume the existence of a function $f \colon \RR^d \to \RR$ satisfying
\begin{align}
    \label{eq.model}
    Y = f(X) + \eps,
\end{align}
where $\eps$ is a centered random variable, independent from $X,$ and which satisfies the uniform exponential condition
\begin{assumption}[Uniform exponential condition]\label{ass.uniform.exponential}
    There exists a couple of constant $\lambda_0, \sigma_0 > 0$ such that
    \begin{align}
        \label{eq.uniform.noise}
        \E\big[\exp(\lambda_0 |\eps|)\big] \leq \sigma_0. 
    \end{align}
\end{assumption}
In particular, any sub-Gaussian and any subexponential random variables satisfy Assumption \ref{ass.uniform.exponential}. Moreover, we assume that $f$ lies in a H\"older ball defined below. 
\begin{definition}[H\"older Class of Functions]
    Let $\beta \in (0, 1], \ F > 0,$ and $L >0.$ We denote by $\mH_\beta(F, L)$ the class of functions $f \colon \RR^d \to \RR$ such that $\|f\|_\infty \leq F,$ and for all $x, y \in \RR^d,$
    \begin{align}\label{eq.holder}
        |f(x) - f(y)| 
        \leq L\|x - y\|^{\beta}.
    \end{align}
\end{definition}
Given an i.i.d.\ $n$-sample $(X_1, Y_1), \dots, (X_n, Y_n)$ following a probability distribution $\P_{\sX, \sY}$ whose marginals satisfy \eqref{eq.model}, the goal is to construct an estimator $\wh f(\cdot, \{(X_i, Y_i)\}_{i=1}^n)$ that minimizes the Mean Squared Error (MSE) $\ell(\wh f) := \E[(\wh f(X) - Y)^2].$ It is well known that under Model \eqref{eq.model}, the MSE is minimised by the Bayes' least squares estimate, $x \mapsto \E[Y | X = x] = f(x).$ It is then standard to shift $\ell_{\Q}(\wh f)$ by the MSE at the Bayes' estimate, which leads to the $L^2(\Q_{\sX})$-norm $\mE_{\Q_{\sX}}(\wh f, f) := \ell_{\Q}(\wh f) - \ell_{\Q}(f) = \E[(\wh f(X) - f(X))^2].$ The latter is a random quantity since it implicitly depends on the sample. In this article, we study bounds in probability for $\mE_{\Q_{\sX}}(\wh f, f).$ The setting of transfer learning marks a departure from the standard risk analysis for nonparametric regression in which the data $\{(X_i, Y_i)\}_{i=1}^n$ follows the same distribution as the testing pair $(X, Y).$ Instead, we consider that the two distributions only share weaker structural characteristics, which we model by the standard Covariate Shift (CS) assumption.
\begin{assumption}[Covariate Shift]\label{ass.covariate.shift}
Let $\P_{\sX, \sY}$ and $\Q_{\sX, \sY}$ be two probability distributions defined on $\RR^d \times \RR.$ We say that $\P_{\sX, \sY}$ and $\Q_{\sX, \sY}$ satisfy the covariate shift assumption if 
\begin{align}
\label{eq.covariate.shift}
\P_{\sY|\sX} = \Q_{\sY|\sX}. \tag{CS}
\end{align}
\end{assumption}
Under \eqref{eq.covariate.shift}, the unknown function $f$ in \eqref{eq.model} is the same under $\P_{\sX, \sY}$ and $\Q_{\sX, \sY}.$ The mismatch takes effect on the level of marginal distributions $\P_{\sX}$ and $\Q_{\sX}.$ Additionally, this means that the noise variable's distribution is invariant across distributions satisfying \eqref{eq.covariate.shift}. 

Since the distribution of $Y$ is entirely characterised by the covariates' distribution $\P_{\sX},$ the unknown function $f \in \mH_\beta(F, L),$ and the noise variable $\eps,$ specifying the distributions pairs $(X, Y)$ becomes redundant. The latter two being fixed,  we now shift our focus towards the covariates' distributions $\P_{\sX}$ and $\Q_{\sX}.$\\

\noindent\textbf{Density ratio exponent.} Intuitively, the prediction risk depends on the source-target pair $(\P_{\sX}, \Q_{\sX}).$ On one hand, the transferless case, mentioned before, refers to the situation $\P_{\sX} = \Q_{\sX}.$ In our setup, this extreme case leads to the well-known minimax optimal rate $n^{-2\beta/(2\beta + d)}.$ Another extreme is reached whenever $\P_{\sX}$ and $\Q_{\sX}$ have disjoint supports. In this case, it is easy to see that the prediction risk for data $\{(X_i, Y_i)\}$ sampled solely from $\P_{\sX},$ is lower bounded by a constant, rendering transfer learning useless for large datasets. Therefore, the interest lies in characterising the behaviour of regression estimates for source-target pairs that realise a compromise between those two extremes. This compromise is quantified by a similarity measure, which, in turn, appears as a modulating parameter in the rates obtained under \eqref{eq.covariate.shift}. One challenge in the recent literature on transfer learning under covariate shift is the design of a similarity measure that induces a meaningful class of source-target pairs. Two of those similarity measures have been mentioned in Section \ref{sec.intro}. This article defines a new similarity measure called \textit{density ratio exponent}.
\begin{definition}[Density Ratio Exponent]
    \label{def.density.ratio.exponent}
    Let $\P_{\sX}, \Q_{\sX} \in \mM(\ol p),$ we say that the couple $(\P_{\sX}, \Q_{\sX})$ has a density ratio exponent $\gamma >0$ if 
    \begin{align}
        \label{eq.density.ratio.exponent}
        \int \frac{q(x)}{p(x)^\gamma}\, dx < \infty,\tag{DRE}
    \end{align}
    where $p$ and $q$ are the respective densities of $\P_{\sX}$ and $\Q_{\sX}.$
\end{definition}
We give examples of source-target pairs that admit a density ratio exponent in Example \ref{ex.exponential.1}. We are now ready to define the class of probability distributions considered in this article.
\begin{definition}
\label{def.class.proba}
Let $\ol p, a_-, r_-, a_+, r_+, \rho$ be positive constants. We define $\mP = \mP(\ol p, a_-, r_-, a_+, r_+, \rho)$ to be the class of distributions $\P_{\sX} \in \mM(\ol p),$ with density $p,$ such that
\begin{enumerate}
    \item [$(i)$]\label{ass.i} $\P_{\sX}$ satisfies a minimal mass property with parameter $a_-, r_-$. That is, for all $x \in \RR^d,$ and all $r \in (0, r_-],$
    \begin{align}
        \label{eq.minimal.mass}
        \P_{\sX}\{\B(x, r)\} \geq a_-p(x)r^d.\tag{mMP}
    \end{align}
    \item [$(ii)$]\label{ass.ii} $\P_{\sX}$ satisfies a maximal mass property with parameters $a_+, r_+.$ That is, for all $x \in \supp(\P)$ and all $r \in (0, r_+],$
    \begin{align}
        \label{eq.maximal.mass}
        \P_{\sX}\{\B(x, r)\} \leq a_+p(x)r^d.\tag{MMP}
    \end{align}
\end{enumerate}
\end{definition}
Assumptions \eqref{eq.minimal.mass} and \eqref{eq.maximal.mass} are satisfied for uniform distributions and typically, for heavy-tailed distributions $\P_{\sX}$ such that $\supp(\P_{\sX}) = \RR^d.$ Standard light-tailed distributions such as Gaussian do not satisfy \eqref{eq.minimal.mass}. If the density $p$ of $\P_{\sX}$ cancels at $x_0 \in \RR^d,$ and $p(x)$ is proportional to $\|x - x_0\|^\alpha$ in a neighbourhood of $x_0$ and some positive $\alpha,$ then $\P_{\sX}$ does not satisfy \eqref{eq.maximal.mass}. A subsequent weakening of the maximal mass property is required in order to extend our analysis to such distributions. Next, we introduce the class of target distributions associated with a density ratio exponent $\gamma$ and a source distribution $\P_{\sX, \sY}.$
\begin{definition}
    Let $\P \in \mP, \ \gamma, \rho > 0$ We denote by $\mQ(\P, \gamma, \rho)$ the class of distributions $\Q \in \mP$ such that $(\P, \Q)$ has a density ratio exponent $\gamma,$ and $\Q_{\sX}$ satisfies a pseudo-moment condition
    \begin{align}
        \label{eq.pseudo.moment}
        \int q(x)^{d/(\rho + d)}\, dx < \infty.\tag{PM}
    \end{align}
\end{definition}
Lemma \ref{lem.pseudo.moment} in the Appendix shows that \eqref{eq.pseudo.moment} is automatically satisfied if $X \sim \P_{\sX}$ admits a generalised moment of order $\rho + \epsilon$ with $\epsilon > 0.$ A generalised moment assumption is sufficient for deriving our results, and \eqref{eq.pseudo.moment} simply avoids obtaining an $\epsilon$ term in the rates. It follows immediately that if $\P_{\sX} \in \mP$ satisfies \eqref{eq.pseudo.moment} for some $\rho > 0,$ then $\P_{\sX} \in \mQ(\P_{\sX}, \gamma, \rho)$ for all $\gamma \leq \rho/(\rho + d).$ Additionally, we show in Lemma \ref{lem.ratio.bound} in the appendix that $\mQ(\P_{\sX}, \gamma, \rho) \subseteq \mQ(\P_{\sX}, \gamma', \rho)$ for all $\gamma' \leq \gamma$ and all $\rho > 0.$ Instructive examples of source-target pair that satisfy the density ratio exponent condition are given by pairs of exponential and pairs of Pareto distributions.
\begin{example}[Exponential distributions]
\label{ex.exponential.1}
For positive $\lambda,$ denote by $\mE(\lambda)$ the exponential distribution with parameter $\lambda.$ Consider $\P_{\sX} = \mE(\lambda_{\P})$ and $\Q_{\sX} = \mE(\lambda_{\Q}).$ Then, the distributions $\P_{\sX}$ and $\Q_{\sX}$ are both in the class $\mP(\lambda, \exp(-\lambda), 1, 2\sinh(\lambda)/\lambda, 1),$ and $\Q_{\sX} \in \mQ(\P_{\sX}, \gamma, \rho)$ for all $\gamma < \lambda_{\Q}/\lambda_{\P}$ and all $\rho > 0.$ A proof of these inclusions can be found in Lemma \ref{lem.example.exponential} in the appendix.
\end{example}

\begin{example}[Pareto distributions]
\label{ex.pareto.1}
    For positive $\alpha,$ denote by $\operatorname{Par}(\alpha)$ the Pareto distribution with location parameter $1$ and scale $\alpha.$ That is, if $\P_{\sX} = \operatorname{Par}(\alpha),$ then its density is $p(x) = x^{-(\alpha + 1)}\mathds{1}(x \geq 1).$ Let $\alpha_{\P}, \alpha_{\Q} > 0.$ Let $\P_{\sX} = \operatorname{Par}(\alpha_{\P}),$ then $\P_{\sX} \in \mP(1, 2, 1/2, a, 1/2)$ with $a = (3/2)^{\alpha_{\P}\vee\alpha_{\Q} + 1}.$ Furthermore, if $\Q_{\sX} = \operatorname{Par}(\alpha_{\Q}),$ then $\Q_{\sX} \in \mQ(\P_{\sX}, \gamma, \rho)$ for all $\gamma < \alpha_{\Q}/(\alpha_{\P} + 1)$ and all $\rho > \alpha_{\Q}.$ A proof of these inclusions can be found in Lemma \ref{lem.pareto} in the appendix.
\end{example}

\noindent\textbf{Local Nearest Neighbours.}\label{sec::method} To address the covariate shift problem, we need to estimate the regression function $f.$ We compare two approaches: the standard $k$-NN and the local $k$-NN. Let $n \in \NN^*$ and a neighbour function $k \colon \RR^d \to \{1, \dots, n\}.$ We define the local $k$-NN estimator as
\begin{align}
    \label{eq.weighted.knn}
    \wh f(x) = \frac 1{k(x)}\sum_{i=1}^{n} Y_i(x)\mathds{1}(i \leq k(x)).
\end{align}
In particular, taking $k$ as a constant independent of $x \in \RR^d$ leads to the standard $k$-NN estimator.
Additionally, we define the two-sample local $k$-NN. Let $\{(X_i, Y_i)\}_{i = 1}^n$ and $\{(X'_i, Y'_i)\}_{i=1}^m$ be, respectively, two sets of $n$ and $m$ i.i.d.\ observations from $\P_{\sX, \sY}$ and $\Q_{\sX, \sY}.$ Given two neighbour functions $k_{\P}$ and $k_{\Q},$ the two-sample local $(k_{\P}, k_{\Q})$-nearest neighbours estimator is defined as
\begin{align}
    \label{eq.two.sample.local.knn}
    \wh f(x) = \frac{1}{k_{\P}(x) + k_{\Q}(x)}\Bigg[\sum_{i=1}^n Y_i(x)\mathds{1}(i \leq k_{\P}(x)) + \sum_{j=1}^m Y'_j(x)\mathds{1}(j \leq k_{\Q}(x))\Bigg].
\end{align}
The standard $(k_{\P}, k_{\Q})$-NN is obtained by choosing constant neighbour functions $k_{\P}$ and $k_{\Q}.$

\section{Theoretical Results} \label{sec.results}
We state here our two theorems. For the sake of simplicity, we omit the multiplicative constants that do not depend on the sample sizes. The complete statements can be found in the appendix, see Theorem \ref{th.rates.standard.knn.proof}, Theorem \ref{th.transfer.local.knn.proof}, Theorem \ref{th.transfer.two.samples.fixed.k.proof}, and Theorem \ref{th.rates.local.k.two.sample.proof}. Let $\gamma, \rho > 0.$ Consider $\P_{\sX} \in \mP, \ \Q_{\sX} \in \mQ(\P_{\sX}, \gamma, \rho),$ and work under the setting of Model \eqref{eq.model} under \eqref{eq.covariate.shift}.\\

\noindent\textbf{Standard one- and two-sample nearest neighbours.}
\begin{theorem}
    \label{th.transfer.standard.knn}
    Let $r_T := \gamma/(\gamma + 1) \wedge 2\beta/(2\beta + d)$ and $r_M := \rho/(2\rho + d)\wedge 2\beta/(2\beta + d).$ Consider arbitrary positive constants $\k_{\P}$ and $\k_{\Q}.$
    \begin{enumerate}
        \item [$(i)$] If $\wh f$ is the one-sample standard $k_{\P}$-NN \eqref{eq.weighted.knn} with $k_{\P}$ defined as
        \begin{align*}
            k_{\P} = \big\lceil \k_{\P} \log(n)^{1 - r_T}n^{r_T}\big\rceil,
        \end{align*}
        then, for all $\tau > 0,$ there exists $N = N(r_T, \k_{\P}) \geq 2d \vee 3,$ and a constant $C_1 > 0,$ such that for all $n \geq N,$ 
        \begin{align*}
            \mE_{\Q_{\sX}}(\wh f, f) &\leq C_1\bigg(\frac{\log n}{n}\bigg)^{r_T}
        \end{align*}
        with probability at least $1 - 2n^{-\tau}$.
        \item [$(ii)$] If $\wh f$ is a two-sample $(k_{\P}, k_{\Q})$-NN \eqref{eq.two.sample.local.knn} with $k_{\P}$ and $k_{\Q}$ defined as
        \begin{align*}
            k_{\P} &= \big\lceil \k_{\P} \log(n + m)^{1 -r_T}n^{r_T}\big\rceil\\
            k_{\Q} &= \big\lceil \k_{\Q}\log(n + m)^{1 - r_M}m^{r_M}\big\rceil.
        \end{align*}
        then, for all $\tau > 0,$ there exists $N = N(r_T, \k_{\P}) \geq 2d, \ M = M(r_M, \k_{\Q}) \geq 2d$ and a constant $C_2 > 0$ such that for all $n \geq N$ and all $m\geq M,$ 
        \begin{align*}
            \mE_{\Q}(\wh f, f) \leq C_2\Bigg(\bigg(\frac{n}{\log(n + m)}\bigg)^{r_T} + \bigg(\frac{m}{\log(n + m)}\bigg)^{r_M}\Bigg)^{-1},
        \end{align*}
        with probability at least $1 - 2n^{-\tau} - 2m^{-\tau}$.
    \end{enumerate}
\end{theorem}
In what follows, the exponent of $\log n/n$ is called the source rate, while the exponent of $\log m/m$ is called the target rate. Theorem \ref{th.transfer.standard.knn} shows that the minimax optimal rate $2\beta/(2\beta + d)$ is obtained, up to a $\log$ factor, for the source rate whenever $\gamma \geq 2\beta/d,$ and for the target rate whenever $\rho/(\rho + d) \geq 2\beta/d.$ Hence, we observe two regimes for each of the source rate and the target rate.\\

\noindent\textbf{Local one- and two-sample nearest neighbours.} We now consider the local nearest neighbours regressor. The aim here is to improve the prediction risk by choosing a neighbour function that balances bias and variance in a pointwise manner. In this context, our choice of neighbours depends on the density of the data. Hence, a good selection of $k$ involves estimating the covariates' densities. We introduce the $\ell$-nearest neighbours density estimator. 
\begin{definition}
    \label{def.density.estimator}
    Let $\P$ be an absolutely continuous distribution on $(\RR^d, \mB(\RR^d))$ and denote by $p$ its density.  Let $\ell \in \{1, \dots, n\}.$ We denote by $\wh p$ the $\ell$-nearest neighbours density estimator of $p$, which expression is given, at $x \in \RR^d,$ by \begin{align}\label{eq.density.estimator}
    \wh p(x):= \frac{\P_n(\B(x, R_{\ell}(x)))}{R_{\ell}(x)^d} = \frac{\ell}{nR_\ell(x)^d},
    \end{align}
    where $\P_n$ denotes the empirical distribution associated with the sample $(X_1, \dots, X_n).$
\end{definition}

We show, in the appendix, that the $\ell$-NN density estimator achieves $\wh p(x) \asymp p(x)$ uniformly over a high-density region $\{x : p(x) \gtrsim \ell/n\}.$ This holds under the minimal and maximal mass assumptions and allows us to obtain the following rates for the local $k$-NN regressor.
\begin{theorem}
    \label{th.transfer.local.knn}
    Let $r_T = 2\beta/(2\beta + d) \wedge \gamma$ and $r_M = \rho/(\rho + d) \wedge 2\beta/(2\beta + d).$ Consider arbitrary positive constant $\k_{\P}$ and $\k_{\Q}.$ 
    \begin{enumerate}
        \item [$(i)$] If $\wh f$ is the one-sample local $k_{\P}$-NN regressor defined by \eqref{eq.weighted.knn} with $k_{\P}$ defined as
        \begin{align*}
            k_{\P}(x) := n\wedge \big\lceil \kappa\log(n)^{d/(2\beta + d)}(n\wh p(x))^{2\beta/(2\beta + d)}\big\rceil \vee \big\lceil \log n \big\rceil,
        \end{align*}
        and $\wh p$ is the $\ell$-NN density estimator of $p$ with $\ell \asymp \log n,$ then, for all $\tau > 1,$ there exists an $N > 0$ and a constant $C_3 > 0$ such that for all $n \geq N,$
        \begin{align*}
            \mE_{\Q}(\wh f, f) &\leq C_3\bigg(\frac{\log n}n\bigg)^{r_T}
        \end{align*}
         with probability at least $1 - 3n^{1- \tau}.$
        \item [$(ii)$] If $\wh f$ is the two-sample local $(k_{\P}, k_{\Q})$-NN regressor defined by \eqref{eq.two.sample.local.knn} with $k_{\P}$ and $k_{\Q}$ defined as
        \begin{align*}
            k_{\P}(x) &:= n\wedge \big\lceil \kappa\log(n + m)^{d/(2\beta + d)}(n\wh p(x))^{2\beta/(2\beta + d)}\big\rceil \vee \big\lceil \log n \big\rceil\\
            k_{\Q}(x) &:= m\wedge \big\lceil \kappa\log(n + m)^{d/(2\beta + d)}(m\wh q(x))^{2\beta/(2\beta + d)}\big\rceil \vee \big\lceil \log m \big\rceil,
        \end{align*}
        where $\wh p$ and $\wh q$ are the $\ell_{\P}$ and $\ell_{\Q}$ nearest neighbour density estimators of $q$ and $p$ respectively, with $\ell_{\P} = \ell_{\Q} \asymp \log(n + m),$ then, for all $\tau > 1,$ there exists $N \geq 2d, \ M \geq 2d$ and $C_4 > 0$ such that for all $\tau > 0, n \geq N, m \geq M,$ we have $n\wedge m \gtrsim \log(n+m),$ and
    \begin{align*}
        \mE_{\Q}(\wh f, f) \leq C_4\Bigg(\bigg(\frac{n}{\log n}\bigg)^{r_T} + \bigg(\frac{m}{\log m}\bigg)^{r_M}\Bigg)^{-1},
    \end{align*}
    with probability at least $1 - 3n^{1- \tau} - 3m^{1- \tau}.$
    \end{enumerate}
\end{theorem}

\begin{wrapfigure}{r}{0.4\textwidth}
  \begin{center} 
    \includegraphics[width=0.39\textwidth]{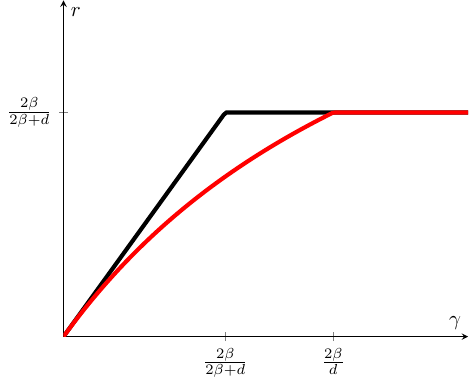}
  \end{center}
\vspace{-.5cm}
  \caption{\label{fig.transfer.rates} Comparison of the rates of the one-sample standard (red) and local (black) regressors. $r$ represents the source rate as a function of $\gamma$ in Theorem \ref{th.transfer.standard.knn} and Theorem \ref{th.transfer.local.knn}.}
\end{wrapfigure}

Comparing Theorem \ref{th.transfer.standard.knn} with Theorem \ref{th.transfer.local.knn}, we can see that the convergence rates of local $k$-NN with $2\beta/(2\beta+d) < \gamma \leq 2\beta/d$ are strictly faster than that of standard $k$-NN. The same remark applies to the target rates whenever $2\beta/(2\beta + d) \leq \rho/(\rho + d) \leq 2\beta/d.$ Figure \ref{fig.transfer.rates} shows how the rates of the standard $k$-NN differ from those of the local $k$-NN as a function of $\gamma.$ This shows that a local adaptation to the covariates' density significantly improves the standard estimator.\\
\indent Interestingly, the best local regressor is obtained by using heavily under-smoothed density estimators, that is, by taking $\ell$ of the order $\log n.$ This choice of $\ell$ ensures that $\wh p/p$ is both upper- and lower-bounded by positive constants on the largest possible subset of $\RR^d.$ It is likely that a non-pointwise analysis would reveal that $\ell = 1$ is an even better choice.\\
\indent In the transferless case, i.e.\ $\Q_{\sX} = \P_{\sX} \in \mP,$ the density ratio exponent is $\gamma = \rho/(\rho+d)$. If $\rho \geq 2\beta$, Theorem \ref{th.transfer.local.knn} yields the convergence rate $(\log n/n)^{2\beta/(2\beta+d)}$, which is optimal up to a factor $\log n$. Theorem 2 of \cite{Kohler2009OptimalGR}, shows that the rate $n^{-2\beta / (2\beta+d)}$ can not be achieved when $\rho < 2\beta$ without additional assumptions. Their counterexample is obtained for covariates' densities of the form $p(x) \propto 1/(1 + \|x\|)^\alpha$ with $\alpha > 1.$ The corresponding distributions satisfy our working assumptions. Therefore, the lower bound holds in our case. This means that the phase transition threshold $\rho/(\rho + d) = 2\beta/(2\beta + d)$ is optimal under our assumptions.\\
\indent The following two examples showcase the results of Theorem \ref{th.transfer.local.knn} in the case of source-target pairs of exponential and Pareto distributions.
\begin{example}[Rates for exponential source-target pairs]
    Example \ref{ex.exponential.1} shows that we can apply Theorem \ref{th.transfer.local.knn} with $\gamma < \lambda_{\Q}/\lambda_{\P}$ and $\rho$ arbitrarily large. We choose $\rho > 2\beta.$ The source rates are optimal if $\lambda_{\Q}/\lambda_{\P} > 2\beta/(2\beta + 1).$ This implies two facts. On the one hand, any data from $\P_{\sX}$ with heavier tails $\lambda_{\P} < \lambda_{\Q}$ always leads to minimax optimal rates. On the other hand, data samples from a source distribution with lighter tails $\lambda_{\P} > \lambda_{\Q}$ still allow for optimal rates if $\lambda_{\P}$ is not too large compared to $\lambda_{\Q}.$ The latter fact is surprising since one would expect that $\P_{\sX}$ would not provide sufficient information in the tails of $\Q_{\sX}.$\\
\end{example}
\begin{example}[Rates for Pareto source-target pairs]
    Let $\alpha_{\P}, \alpha_{\Q} > 0.$ Let $\P_{\sX} = \operatorname{Par}(\alpha_{\P}),$ and $\Q_{\sX} = \operatorname{Par}(\alpha_{\Q}).$ Example \ref{ex.pareto.1} shows that Theorem \ref{th.transfer.local.knn} is applicable and we obtain the rates
    \begin{align*}
        r_T = \frac{2\beta}{2\beta + 1} \wedge \frac{\alpha_{\Q}}{\alpha_{\P} + 1}, \text{ and } r_M = \frac{2\beta}{2\beta + 1}\wedge \frac{\alpha_{\Q}}{\alpha_{\Q} + 1}.
    \end{align*}
    Optimal source rates are obtained for $\alpha_{\Q} > (\alpha_{\P} + 1)2\beta/(2\beta + 1)$ and optimal target rates are obtained for $\alpha_{\Q} > 2\beta.$ Once again, optimal source rates are obtainable even if the source distribution has lighter tails than the target. Moreover, if $\alpha_{\Q} \geq 2\beta,$ then any choice of $\alpha_{\P} \leq 2\beta$ will lead to optimal rates.
\end{example}

\section{Outline of the proof}
\label{sec.proof.outline}

The proof relies on a uniform pointwise analysis of the error associated with the local nearest neighbours regressor $\wh f.$ Considering $\ol f \colon x \mapsto \E[\wh f(x)| X^n],$ the following bias-variance decomposition holds for all $x \in \RR^d,$
\begin{align*}
    \big|\wh f(x) - f(x)\big| &\leq \underbrace{\big|\wh f(x) - \ol f(x)\big|}_{\text{Variance}} + \underbrace{\big|\ol f(x) - f(x)\big|}_{\text{Bias}}.
\end{align*}
Using the smoothness and the boundedness of $f \in \mH_\beta(L, F),$ the bias term is bounded by $LR_{k(x)}(x)^{\beta}\wedge 2F.$ The nearest neighbour radius $R_{k(x)}(x)$ is then related to a deterministic quantity that depends only on the design distribution. Let $h>0$ and define the function $\zeta_h \colon x \mapsto \inf\{t > 0: \P\{\B(x, t)\} \geq h\}.$ Picking $h = h(x) \asymp k(x)/n, \ k \colon x \mapsto \{\lceil \log n\rceil, \dots, n\},$ and exploiting the fact that the set of closed balls in $\RR^d$ is Vapnik-Chervonenkis (VC) subgraph, Proposition \ref{prop.neighbours.distance.bound.zeta} shows that
\begin{align*}
    \sup_{x \in \RR^d}\frac{R_{k(x)}(x)}{\zeta_{h(x)}(x)} \leq 1
\end{align*}
holds with high probability. The minimal mass property of $\P_{\sX}$ implies that $\zeta_h(x) \lesssim (k(x)/np(x))^{1/d},$ for all $x$ in $\Delta_n := \{x \in \RR^d: p(x) \gtrsim k(x)/n\}.$ Hence, with high probability, for $x \in \Delta_n,$
\begin{align*}
    |f(x) - \ol f(x)| &\lesssim \bigg(\frac{k(x)}{np(x)}\bigg)^{\beta/d},
\end{align*}
Next, we deal with the variance term, which reads
\begin{align}
    \label{eq.variance.pointwise}
	|\wh f(x) - \ol f(x)| &= \frac 1{k(x)}\sum_{i=1}^{k(x)} \big|Y_i(x) - f(X_i(x))\big|.
\end{align}
Given $x \in \RR^d$ and a realisation of $X^n = (X_1, \dots, X_n),$ the uniform exponential assumption \eqref{eq.uniform.noise} together with Theorem 12.2 in \cite{Biau2015Lectures} provides an exponential bound on the tail of the sum in \eqref{eq.variance.pointwise}. Conditionally on the sample $X^n,$ any given $x \in \RR^d$ can be associated with a permutation $\sigma_x$ of $\{1, \dots, n\}$ that satisfies $(X_1(x), \dots, X_n(x)) = (X_{\sigma_x(1)}, \dots, X_{\sigma_x(n)}).$ Let the $\mW = \{\sigma_x : x \in \RR^d\}$ be the set of permutations that are obtainable when $x$ describes $\RR^d.$ The Milnor-Thom theorem (see Theorem 12.2 in \cite{Biau2015Lectures}) shows that the cardinality of $\mW$ grows polynomially in $n.$ Upon applying the union bound and taking the expectation with respect to the sample, we are left with
\begin{align*}
    \big|\wh f(x) - \ol f(x)\big| \lesssim \sqrt{\frac{\log n}{k(x)}},
\end{align*}
uniformly over $x \in \RR^d$ and with high probability.\\

\noindent\textbf{Bounding the prediction error of the two-sample estimator} Consider now $\wh f$ to be a two-sample $(k_{\P}, k_{\Q})$-NN. One can rewrite
\begin{align*}
    |\wh f(x) - f(x)| &= \frac{k_{\P}(x)}{k_{\P}(x) + k_{\Q}(x)}|\wh f_{\P}(x) - f(x)| + \frac{k_{\Q}(x)}{k_{\P}(x) + k_{\Q}(x)}|\wh f_{\Q}(x) - f(x)|,
\end{align*}
where $f_{\P}$ and $f_{\Q}$ are, respectively, a $k_{\P}$ and a $k_{\Q}$ one-sample nearest neighbour regressors associated with an $n$-sample from $\P_{\sX, \sY}$ and an $m$-sample from $\Q_{\sX, \sY}.$ We apply the previously derived high-probability bounds to both terms in the above decomposition. To do so, we define two sets $\Delta_{\P} := \{x : p(x) \gtrsim \log (n + m)/n\}$ and $\Delta_{\Q} := \{x : q(x) \gtrsim \log (n + m)/m\},$ and we bound the prediction error as follows
\begin{multline}
    \label{eq.two.sample.decomposition.proof.outline}
    \mE_{\Q_{\sX}}(\wh f, f) \lesssim \int_{\Delta_{\P}\cap \Delta_{\Q}} \frac{\log(n + m)}{k_{\P}(x) + k_{\Q}(x)} \, d\Q_{\sX}(x)\\
    + \int_{\Delta_{\P}\cap\Delta_{\Q}} \frac{1}{k_{\P}(x) + k_{\Q}(x)}\Bigg[k_{\P}(x)\bigg(\frac{k_{\P}(x)}{np(x)}\bigg)^{2\beta/d} + k_{\Q}(x)\bigg(\frac{k_{\Q}(x)}{mq(x)}\bigg)^{2\beta/d}\Bigg]\, d\Q_{\sX}(x)\\
    +\int_{\Delta_{\P}^c\cap \Delta_{\Q}}\frac{k_{\P}(x)}{k_{\P}(x) + k_{\Q}(x)}\, d\Q_{\sX}(x) + \int_{\Delta_{\P}\cap \Delta_{\Q}^c}\frac{k_{\Q}(x)}{k_{\P}(x) + k_{\Q}(x)}\, d\Q_{\sX}(x)\\
    + \Q_{\sX}\{\Delta_{\P}^c \cap \Delta_{\Q}^c\}.
\end{multline}
We now take $k_{\P}$ and $k_{\Q}$ as in Theorem \ref{th.transfer.local.knn}. It can be shown that in the regime $n \wedge m \gtrsim \log(n + m),$ for $\ell_{\P} \asymp \log (n + m),$ the maximal mass property implies the existence of two constants $0 < c \leq C < \infty$ such that $c \leq \wh p(x)/p(x) \leq C$ for all $x \in \Delta_{\P}.$ A similar result holds for $\wh q$ and $\ell_{\Q} \asymp \log (n + m).$ This, and some algebra, shows that the first two terms in \eqref{eq.two.sample.decomposition.proof.outline} are equivalent to $\log(n + m)/(np(x) + mq(x)).$ Integrating the latter over $\Delta_{\P}\cap \Delta_{\Q}$ leads to
\begin{multline*}
     \int_{\Delta_{\P}\cap\Delta_{\Q}}\bigg(\frac{\log(n+m)}{np(x) + mq(x)}\bigg)^{2\beta/(2\beta + d)}\\
     \lesssim \int_{\Delta_{\P}}\bigg(\frac{\log(n+m)}{np(x)}\bigg)^{2\beta/(2\beta + d)}q(x)\, dx \wedge \int_{\Delta_{\Q}}\bigg(\frac{\log(n+m)}{mq(x)}\bigg)^{2\beta/(2\beta + d)}q(x)\, dx.
\end{multline*}
Denote by $R^{\P}_{\ell_{\P}}(x)$ the $\ell_{\P}$ nearest neighbour's distance to $x$ among the data sampled from $\P_{\sX, \sY}$ and by $R^{\Q}_{\ell_{\Q}}(x)$ the $\ell_{\Q}$ nearest neighbour's distance to $x$ among the data sampled from $\Q_{\sX, \sY}.$ The maximal mass property enables us to show that $R^{\P}_{\ell_{\P}}(x)$ is lower-bounded by a positive constant for all $x \in \Delta_{\P}^c.$ Hence, using $\ell_{\P} \asymp \ell_{\Q} \asymp \log(n + m),$ in the regime $n\wedge m \gtrsim \log(n + m),$
\begin{align*}
    \int_{\Delta_{\P}^c\cap \Delta_{\Q}}\frac{k_{\P}(x)}{k_{\P}(x) + k_{\Q}(x)}\, d\Q_{\sX}(x) &\lesssim \int_{\Delta_{\P}^c\cap\Delta_{\Q}} 1 \wedge \Bigg(\frac{R^{\Q}_{\ell_{\Q}}(x)^d}{R^{\P}_{\ell_{\P}}(x)^d}\Bigg)^{2\beta/(2\beta + d)}\, d\Q_{\sX}(x)\\
    &\lesssim \int_{\Delta_{\P}^c\cap\Delta_{\Q}} 1\wedge\bigg(\frac{\log (n + m)}{mq(x)}\bigg)^{2\beta/(2\beta + d)}\, d\Q_{\sX}(x)\\
    &\lesssim \int_{\Delta_{\P}^c}q(x)\, dx \wedge \int_{\Delta_{\Q}} \bigg(\frac{\log (n + m)}{mq(x)}\bigg)^{2\beta/(2\beta + d)}q(x)\, dx.
\end{align*}
A similar analysis leads to
\begin{align*}
    \int_{\Delta_{\P}\cap \Delta_{\Q}^c}\frac{k_{\Q}(x)}{k_{\P}(x) + k_{\Q}(x)}\, d\Q_{\sX}(x) &\lesssim \int_{\Delta_{\Q}^c}q(x)\, dx\wedge \int_{\Delta_{\P}}\bigg(\frac{\log (n + m)}{np(x)}\bigg)^{2\beta/(2\beta + d)}q(x)\, dx.
\end{align*}
Finally, the result of Theorem \ref{th.transfer.local.knn} $(ii)$ are obtained by bounding each of the terms, including $\Q_{\sX}\{\Delta_{\P}^c\cap\Delta_{\Q}^c\},$ using that, for non-negative $f,$ if $\int f(x)^b\, dx < \infty$ and $1 \geq a \geq b \geq 0,$ then
\begin{align*}
    \int_{\{x : f(x) \geq \delta\}} f(x)^a\, dx \leq \delta^{a - b}\int f(x)^b\, dx,\text{ and }\int_{\{x : f(x) < \delta\}} f(x)\, dx \leq \delta^{1 - b}\int f(x)^b\, dx,
\end{align*}
together with \eqref{eq.pseudo.moment} and \eqref{eq.density.ratio.exponent}.

\section{Conclusion}\label{sec.conclusion}

We have introduced a novel similarity measure for transfer learning under covariate shift called density ratio exponent. In particular, the density ratio exponent quantifies the relative tail decays of the source and target distributions. We have also introduced the local $k$-nearest neighbours regressor and derived its convergence rates under smoothness assumption and covariate shift. Our results indicate that the standard $k$-NN regressor is outperformed by its local counterpart, demonstrating the advantages of a local choice of $k.$ Our rates account for the tails of both source and target distributions and are optimal if the target admits moments of order $\rho > 2\beta$ and the density ratio exponent is larger than $2\beta/(2\beta + d).$ Existing literature shows that optimal rates are unobtainable for $\rho \leq 2\beta$ and $\gamma < 2\beta/(2\beta + d).$ It remains unclear whether the rates we obtain in those regimes are minimax optimal. Interestingly, minimax optimal rates can be attained even when the likelihood ratio $q(x)/p(x)$ explodes. Our results are obtained under the strong maximal mass assumption. A weakening of the latter will reveal more intricate dynamics for the convergence rates under covariate shift. The authors stress the interest of a weakening that would extend the analysis to distributions that admit zeroes in $\RR^d.$ Such weakening and derivations of corresponding minimax lower bounds are left for future research.

\appendix

\section{Preliminary Results} \label{sec::TheoreticalResults}

\subsection{Technical results}

Let $\gamma, \rho > 0$ and consider $\P \in \mP$ and $\Q \in \mQ(\P, \gamma, \rho).$ We first derive some properties related to the density ratio exponent and the pseudo-moment condition. Denote $\mT_{\Q} \colon t \mapsto \int q(x)^{1 - t}\, dx$ and $\mT_{\P} \colon t \mapsto \int q(x)/p(x)^t\, dx.$

\begin{lemma}\label{lem.ratio.bound}
    \begin{enumerate}
        \item [$(i)$] For any $0 < x < y \leq \gamma,$ there holds
            \begin{align*}
                \mT_{\P}(x)^{1/x} \leq \mT_{\P}(y)^{1/y}.
            \end{align*}
            In particular, $\mT_{\P}(x) \in \RR$ for all $x \in [0, \gamma].$
        \item [$(ii)$] For any $0 < x < y \leq \rho/(\rho + d),$ there holds
            \begin{align*}
                \mT_{\Q}(x)^{1/x} \leq \mT_{\Q}(y)^{1/y}.
            \end{align*}
            In particular, $\mT_{\Q}(x) \in \RR$ for all $x \in [0, \rho/(\rho + d)].$
    \end{enumerate} 
\end{lemma}
\begin{proof}
   Let $X \sim \Q.$ For all $x \in [0, \gamma],$ the function $t \mapsto t^{\gamma/x}$ is convex on $\RR_+.$ By Jensen's inequality, we have 
   \begin{align*}
       \mT_{\P}(x) = \big(\E[p(X)^{-x}]^{\gamma/x}\big)^{x/\gamma} \leq \E[p(X)^{-\gamma}]^{x/\gamma} = \mT_{\P}(\gamma)^{x/\gamma}.
   \end{align*}
   The right-hand side of the latter inequality is $<\infty$ by assumption. This proves the second claim of $(i).$ The first claim of $(i)$ is proven by taking $\gamma = y$ in the previous display. The same computations with $p = q$ and $\gamma = \rho/(\rho + d)$ lead to $(ii).$
\end{proof}

The next result shows that Assumption \ref{def.class.proba} $(iii)$ is satisfied by any random variable that admits a moment of order $\rho + \eps,$ with arbitrary $\eps > 0.$
\begin{lemma}\label{lem.pseudo.moment}
If $\eps >0,$ and $X \sim \Q$ admits a generalized moment of order $\rho + \eps,$ then $\Q$ satisfies \eqref{eq.pseudo.moment} with constant $\rho.$
\end{lemma}

\begin{proof}
Let $\tau = d/(\rho + d).$ By H\"{o}lder's inequality, we have
\begin{align*}
& \int  q(x)^{\tau} \, dx 
   = \int (1 + \|x\|^{\rho + \eps})^{\tau} q(x)^{\tau} (1 + \|x\|^{\rho + \eps})^{-\tau} \, dx
\\
& \leq \biggl( \int \bigl( (1 + \|x\|^{\rho + \eps})^{\tau} q(x)^{\tau} \bigr)^{1/\tau} \, dx \biggr)^{\tau} \cdot
           \biggl( \int (1 + \|x\|^{\rho + \eps})^{- \tau / (1 - \tau)} \, dx \biggr)^{1-\tau}
\\
& \leq (1 + M)^{\tau} \biggl( \int (1 + \|x\|^{\rho + \eps})^{- \tau / (1 - \tau)} \, dx \biggr)^{1-\tau},
\end{align*}
where $M := \int\|x\|^{\rho + \eps}\, \Q(dx).$ Since $\tau/(1 - \tau) = d/\rho,$ we have $\tau(\rho + \eps) / (1 - \tau) > d$ and consequently
\begin{align*}
\int (1 + \|x\|^{\rho + \eps})^{- \tau / (1 - \tau)} \, dx < \infty,
\end{align*}
which proves the claim.
\end{proof}

\begin{proposition}\label{prop.covering.number}
Let $\mB := \{\B(x, r) : x \in \RR^d, r > 0 \}$. 
The VC dimension of $\mB$ is $2d + 1$. 
Moreover, there exists a universal constant $K$ such that for any $\eps \in (0, 1)$, 
\begin{align*}
\mathcal{N}(\mathds{1}_{\mB}, \|\cdot\|_{L_1(\Q)}, \eps)
\leq K (2d+1) (4e)^{2d+1} \eps^{-2d}
\end{align*}
holds for any probability measure $\Q,$ where $\mathds{1}_{\mB} := \{\mathds{1}(\cdot \in B), B \in \mB\}.$
\end{proposition}

\begin{proof}
The first result of VC dimension follows from Example 2.6.1 in \cite{MR1385671}. The second result on the covering number follows from Theorem 9.2 in \cite{MR2724368}.
\end{proof}

\begin{lemma}[Bernstein's Inequality]
    Let $X_1, \dots, X_n$ be $n$ i.i.d.\ random variables with values in $\RR.$ If $\PP\{|X_1| \leq C\} = 1$ for some $C > 0,$ and $\E[X_1] = \mu,$ then, for all $t > 0,$
    \begin{align}
        \label{eq.Bernstein}
        \PP\bigg\{\Big|\frac 1n \sum_{i=1}^n X_i - \mu\Big| \geq t\bigg\} \leq 2\exp\bigg(-\frac{nt^2}{2\sigma^2 + 2Ct/3}\bigg),
    \end{align}
    where $\sigma^2 = \Var(X_1).$
\end{lemma}

Let $\mM$ be the set of all probability distributions on $(\RR^d, \mB(\RR^d)).$
\begin{definition}
    For $\P \in \mM,$ and $h > 0,$ we define the function $\zeta_h^{\P}$ as 
    \begin{align}
        \label{eq.zeta.definition}
        \zeta_h^{\P}\colon x\mapsto \inf\Big\{r > 0: \P\{\B(x, r) \geq h\}\Big\}.
    \end{align}
    When the considered distribution $\P$ is clear from the context, we will drop the $\P$ exponent in the notations and write $\zeta_h$ instead.
\end{definition}
Next, we prove a uniform bound for $R_k.$ Let $\P \in \mM, \ n \geq 2d, \ X_1, \dots, X_n$ be $n$ i.i.d.\ random variables with distribution $\P.$ Let $\delta, h \in (0, 1).$
\begin{proposition}
    \label{prop.Bernstein.general}
    Let $K_0 = 2K(2d + 1)(4e)^{2d+1}$ and $K$ the constant defined in Proposition \ref{prop.covering.number}. The two following assertions hold.
    \begin{enumerate}
        \item [$(i)$] If $\lambda \in (0, 1)$ and the integer $n$ satisfies
            \begin{align}
                \label{eq.Bernstein.nk.condition}
                n &\geq \frac{4d(1 + \lambda/3)}{\lambda^2h}\log\bigg(\Big(\frac{K_0}{\delta}\Big)^{1/(2d)}\frac{2}{(1 - \lambda)h}\bigg),
            \end{align}
            then
            \begin{align*}
                \PP\Bigg\{\inf_{x \in \RR^d}\frac{1}{nh}\sum_{i=1}^n \mathds{1}(X_i \in \B(x, \zeta_h(x))) < \frac{(1 - \lambda)}{2}\Bigg\} \leq \delta.
            \end{align*}
        \item [$(ii)$] If $\lambda \in \RR_+^*$ and the integer $n$ satisfies
            \begin{align}
                \label{eq.Bernstein.nk.lower.condition}
                n &\geq \frac{4d(1 + \lambda/3)}{\lambda^2h}\log\bigg(\Big(\frac{K_0}{\delta}\Big)^{1/(2d)}\frac{1}{(1 + \lambda)h}\bigg),
            \end{align}
            then,
            \begin{align*}
                \PP\Bigg\{\sup_{x \in \RR^d}\frac{1}{nh}\sum_{i=1}^n \mathds{1}(X_i \in \B(x, \zeta_h(x))) > 2(1 + \lambda)\Bigg\} \leq \delta.
            \end{align*}
    \end{enumerate}
\end{proposition}

\begin{proof}
    Let $z^m := (z_1, \dots, z_m) \in \RR^{d\times m}$ be an arbitrary vector of $m$ different points, denote $S(x, h) := \B(x, \zeta_h(x)),$ and define the set
    \begin{align*}
        \mI(h) := \bigg\{\mathds{1}(\cdot \in S(x, h))\bigg\}.
    \end{align*}
    Consider the event
    \begin{align*}
        A_n(t, z^m) := \bigcap_{j=1}^m\bigg\{\Big|\frac 1n \sum_{i=1}^n \mathds{1}\big(X_i \in S(z_j, h)\big) - h \Big|\leq t\bigg\}.
    \end{align*}
    Let $X \sim \P.$ For all $x \in \RR^d,$ we have
    \begin{enumerate}
        \item [$(i)$] $\E[\mathds{1}(X \in S(x, h))] = \P_{\sX}\{S(x, h)\} = \P_{\sX}\{\B(x, \zeta_h(x))\} = h$
        \item [$(ii)$] $\Var(\mathds{1}(X \in S(x, h))) = h(1 - h) \leq h.$
        \item [$(iii)$] $0 \leq \mathds{1}(X \in S(x, h)) \leq 1.$
    \end{enumerate}
    By the union bound and Bernstein's inequality \eqref{eq.Bernstein}, 
    \begin{align}
        \label{eq.Bernstein.union.bound}
        \PP\big\{A_n(t, z^m)^c\big\} &\leq \sum_{j=1}^m \PP\bigg\{\Big|\frac 1n \sum_{i=1}^n \mathds{1}\big(X_i \in S(z_j, h)\big) - h \Big|\geq t\bigg\} \leq 2m\exp\bigg(-\frac{nt^2}{2h + 2t/3}\bigg).
    \end{align}
    Let $\P_n := n^{-1}\sum_{i=1}^n \delta_{X_i}$ be the empirical distribution of the sample $X_1, \dots, X_n,$ and let $\eps \in (0, 1).$ Applying Proposition \ref{prop.covering.number} ensures that there exists an $\eps$-net over $(\mB, \|\cdot\|_{L_1(\P_n)})$ with finite cardinality $M = M(X^n),$ with $\|M\|_\infty \leq K_0\eps^{-2d}.$ As $\mI(h) \subseteq \mB,$ there exists an $\eps$-net over $(\mI(h), \|\cdot\|_{L_1(\P_n)})$ with finite (random) cardinality $\ol m = \ol m(X^n)$ upper bounded by $\|M\|_{\infty}.$  Let $Z_1, \dots Z_{\ol m}$ be such that this $\eps$-net is the set $\big\{S(Z_j, h), 1\leq j\leq \ol m\big\}.$ Since the upper bound on $m$ does not depend on the sample, the upper bound \eqref{eq.Bernstein.union.bound} is also valid when $z^m$ is replaced by $Z^{\ol m}.$ Hence, using the covering bound of Proposition \ref{prop.covering.number} yields
    \begin{align*}
        \PP\big\{A_n(t, Z^{\ol m})^c\big\}&\leq K_0\eps^{-2d}\exp\bigg(-\frac{nt^2}{2h + 2t/3}\bigg).
    \end{align*}
    We now work on $A_n(t, Z^{\ol m}).$ The $\eps$-net property of the $Z_j$'s ensures that for all $x \in \RR^d,$ there exists $j \in [\ol m],$ such that
    \begin{align*}
        \big\|\mathds{1}(\cdot \in S(x, h)) - \mathds{1}(\cdot \in S(Z_j, h))\big\|_{L_1(\P_n)} \leq \eps,
    \end{align*}
    and, on $A_n(t, Z^{\ol m}),$ it follows that
    \begin{align*}
        \Big|\frac 1n \sum_{i=1}^n \mathds{1}(X_i \in S(x, h)) - h\Big| \leq \big\|\mathds{1}&(\cdot \in S(x, h)) - \mathds{1}(\cdot \in S(Z_j, h))\big\|_{L_1(\P_n)}\\
        &+ \Big|\frac 1n \sum_{i=1}^n \mathds{1}(X_i \in S(Z_j, h)) - h\Big|,
    \end{align*}
    where the first term is upper bounded by $\eps,$ while the definition of $A_n(t, Z^{\ol m})$ ensures that the second term is upper bounded by $t.$ The point $x \in \RR^d$ has been taken arbitrarily in $\RR^d,$ hence,
    \begin{align}
        \label{eq.two.sided.uniform.Bernstein}
        \sup_{x \in \RR^d}\Big|\frac 1n \sum_{i=1}^n \mathds{1}(X_i \in S(x, h)) - h\Big| \leq \eps + t.
    \end{align}
    We now proceed to prove the first claim. In particular, from \eqref{eq.two.sided.uniform.Bernstein}, it holds that
    \begin{align}
        \label{eq.0000}
        \inf_{x \in \RR^d}\frac 1n \sum_{i=1}^n \mathds{1}(X_i \in S(x, h)) \geq h -
        \eps - t.
    \end{align}
    Let $\delta, \lambda \in (0, 1).$ We pick $t = \lambda h,$ and
    \begin{align*}
        \eps = \delta^{-1/(2d)}K_0^{1/(2d)}\exp\bigg(-\frac{n\lambda^2h}{4d(1 + \lambda/3)}\bigg).
    \end{align*}
    One can check that for all $n$ satisfying
    \begin{align*}
        n \geq \frac{4d(1 + \lambda/3)}{\lambda^2h}\log\bigg(\frac{2}{(1 - \lambda)h}\Big(\frac{K_0}{\delta}\Big)^{1/(2d)}\bigg),
    \end{align*}
    it holds that $\eps \leq (1 - \lambda)h/2.$ Finally, plugging these values in \eqref{eq.0000} shows that, on $A_n(t, Z^{\ol m}),$
    \begin{align*}
        \inf_{x \in \RR^d}\frac 1{n} \sum_{i=1}^n \mathds{1}(X_i \in S(x, h)) \geq \frac{(1 - \lambda)h}2.
    \end{align*}
    This proves that
    \begin{align*}
        \PP\Bigg\{\inf_{x \in \RR^d}\frac 1{nh} \sum_{i=1}^n \mathds{1}(X_i \in S(x, h)) < \frac{(1 - \lambda)}2\Bigg\} < \PP\big\{A_n(t, Z^{\ol m})^c\big\},
    \end{align*}
    and it can be checked that $\PP\big\{A_n(t, Z^{\ol m})^c\big\} \leq \delta,$ which finishes the proof of the first claim. We now prove the second claim, which is a slight variation of the proof of the first one. A consequence of \eqref{eq.two.sided.uniform.Bernstein} is that
    \begin{align*}
        \sup_{x \in \RR^d}\frac 1n \sum_{i=1}^n \mathds{1}(X_i \in S(x, h)) \leq h + \eps + t.
    \end{align*}
    We then pick $t = \lambda h$ with $\lambda \in (0, \infty)$ and $\eps$ as previously. In particular, it holds that for all 
    \begin{align*}
        n \geq \frac{4d(1 + \lambda/3)}{\lambda^2h}\log\Bigg(\frac{1}{(1 + \lambda) h}\bigg(\frac{K_0}{\delta}\bigg)^{1/(2d)}\Bigg),
    \end{align*}
    we have $\eps \leq (1 + \lambda)h.$ Hence, for such $n$,
    \begin{align*}
        \sup_{x \in \RR^d}\frac 1n \sum_{i=1}^n \mathds{1}(X_i \in S(x, h)) \leq 2(1 + \lambda)h,
    \end{align*}
    which proves that
    \begin{align*}
        \PP\Bigg\{\sup_{x \in \RR^d}\frac 1{nh} \sum_{i=1}^n \mathds{1}(X_i \in S(x, h)) < 2(1 + \lambda)\Bigg\} < \PP\big\{A_n(t, Z^{\ol m})^c\big\} \leq \delta.
    \end{align*}
    This finishes the proof.
\end{proof}

One can check that both the conclusions of Proposition \ref{prop.Bernstein.general} are true whenever \eqref{eq.Bernstein.nk.condition} is satisfied. From now on, for technical reasons, we assume that the constant $K$ in Proposition \ref{prop.covering.number} is greater than $1.$ This assumption only potentially worsens the covering number bound. We now consider an integer $k \in \{1, \dots, n\}.$

\begin{proposition}
    \label{prop.neighbours.distance.bound.zeta}
    Let $\tau > 0.$ If $k \geq \lceil \log n\rceil,$ and $h_+ = h_+(k) := 2k/((1 - \lambda^+_\tau)n)$ with
    \begin{align*}
        \lambda_\tau^+ := \frac{4\big(2d + \tau + \log(K_0)/\log(2d))\big) + 9}{4\big(2d + \tau + \log(K_0)/\log(2d))\big) + 12},
    \end{align*} 
    then,
    \begin{align*}
        \PP\Bigg\{\sup_{x \in \RR^d}\frac{R_k(x)}{\zeta_{h_+}(x)} \geq 1\Bigg\} \leq n^{-\tau}.
    \end{align*}
\end{proposition}
We further define 
\begin{align}
    \label{eq.def.btau}
    B_\tau := \frac{2}{1 - \lambda_\tau^+},
\end{align}
so that $h_+ = B_\tau k/n.$

\begin{proof}
    Applying Proposition \ref{prop.Bernstein.general} $(i)$ with $\delta = n^{-\tau},$ and $h = 2k/((1 - \lambda)n)$ with arbitrary $\lambda \in (0, 1),$ yields that for all $n$ satisfying
    \begin{align}
        \label{eq.n.condition.nearest.neighbours.bound.zeta}
        n \geq \frac{2d(1 + \lambda/3)(1 - \lambda)n}{\lambda^2k}\log\Big(K_0^{1/(2d)}n^{1 + \tau/(2d)}k^{-1}\Big),
    \end{align}
    it holds that
    \begin{align}
        \label{eq.bound.zeta.neighbour.hp}
        \PP\Bigg\{\inf_{x \in \RR^d}\frac{1}{n}\sum_{i=1}^n \mathds{1}(X_i \in \B(x, \zeta_{h_+}(x))) <\frac kn\Bigg\} \leq n^{-\tau},
    \end{align}
    where $K_0$ is defined in Proposition \ref{prop.Bernstein.general}. The high-probability bound \eqref{eq.bound.zeta.neighbour.hp} implies that
    \begin{align*}
        \PP\Bigg\{\sup_{x \in \RR^d}\frac{R_k(x)}{\zeta_{h_+}(x)} \geq 1\Bigg\} \leq n^{-\tau}.
    \end{align*}
    We now check that the right-hand side of \eqref{eq.n.condition.nearest.neighbours.bound.zeta} is always smaller than $2d$. Inequality \eqref{eq.n.condition.nearest.neighbours.bound.zeta} is equivalent to
    \begin{align}
        \label{eq.1111}
        \frac{(1 + \lambda/3)(1-\lambda)}{\lambda^2k}\log\Big(K_0n^{2d + \tau}k^{-2d}\Big) \leq 1.
    \end{align}
    For $n \geq k \geq \log n,$ using the expression of $K_0,$ the assumption that $K \geq 1,$ and $d\geq 1,$ we can check that for $n^\tau \geq (384e^2)^{-1},$ it holds that $\log(K_0n^{2d + \tau}k^{-2d})\geq 0.$ The latter then holds for any $\tau > 0$ and any $n \geq 1.$ Hence, we can assert that for \eqref{eq.1111} to hold, it is sufficient to have
    \begin{align*}
        \frac{(1 + \lambda /3)(1 - \lambda)}{\lambda^2}\log\Big(K_0n^{2d + \tau}\Big) \leq \log n.
    \end{align*}
    Using $a\log(b) = \log(b^a),$ exponentiating on both sides and rearranging the terms leads to 
    \begin{align*}
        1 \leq n^{\tfrac{\lambda^2}{(1 + \lambda/3)(1 - \lambda)} - 2d - \tau - \log(K_0)/\log(n)}.
    \end{align*}
    For the latter to hold, it is sufficient that the exponent be nonnegative. Since $n \geq 2d,$ a lower bound for the exponent in the above display is $\lambda^2/((1 + \lambda/3)(1 - \lambda)) - 2d - \tau - \log(K_0)/\log(2d).$ Rearranging the expression of the derived lower bound gives the sufficient condition under the form of a quadratic inequation
    \begin{align*}
        \lambda^2\Big(1 + \frac{a}{3}\Big) + \frac{2a}{3}\lambda - a.
    \end{align*}
    where $a:= 2d + \tau + \log(K_0)/\log(2d).$ It is straightforward to find that the positive root $z_+$ satisfies
    \begin{align*}
        z_+ = \frac{a}{a + 3}\bigg(\sqrt{4 + \frac 9a} - 1\bigg) \leq \frac{4a + 9}{4a + 12} = \lambda_\tau^+ < 1,
    \end{align*}
    where the inequality follows Bernoulli's inequality. This shows that with our choice of $\lambda_\tau^+,$ \eqref{eq.n.condition.nearest.neighbours.bound.zeta} holds for all $n \geq 2d.$ Hence, for all $k \geq \log n,$
    \begin{align*}
        \PP\Bigg\{\sup_{x \in \RR^d}\frac{R_k(x)}{\zeta_{h_+}(x)} \geq 1\Bigg\} \leq n^{-\tau}.
    \end{align*}
\end{proof}

Let $\P \in \mP.$ Let $X \sim \P,$ and  $Y = f(X) + \eps,$ where $\eps$ is an $\RR$-valued random variable that is independent from $X.$ Assume as well that $X$ satisfies the uniform noise condition \eqref{eq.uniform.noise}. Let $(X_1, Y_1), \dots, (X_n, Y_n)$ be $n$ i.i.d.\ copies of $(X, Y).$ For $k \in \NN$ and $x \in \RR^d,$ we denote by $R_k^{\P}(x) := \|x - X_k(x)\|.$ We define the events
\begin{align*}
    U_n^{\P}(k) := \Bigg\{\sup_{x \in \RR^d} \frac{R_k^{\P}(x)}{\zeta^{\P_{\sX}}_{h_+(k)}(x)} \leq 1\Bigg\}, \text{ and }U_n^{\P} := \bigcap_{k = \lceil \log n\rceil}^n U^{\P}_n(k).
\end{align*}

\begin{remark}
    \label{rem.bound.variable.k}
    While Proposition \ref{prop.neighbours.distance.bound.zeta} shows that $R_k(x) \leq \zeta_{h_+}(x)$ uniformly with high probability for each individual choice of $k \geq \log n,$ it can easily be extended to the case where $k = k(x) \in \{1, \dots, n\}$ depends on $x \in \RR^d.$ The fact that $k$ can now change values depending on $x$ demands one to pay the price that the $\sup$ in the probability is not on the whole $\RR^d$ anymore, but only on a subset $\Theta_n^*(k)$ defined as follows
    \begin{align*}
        \Theta_n(k) := \bigg\{x \in \RR^d: k(x) \geq \lceil\log n \rceil\bigg\}.
    \end{align*}
    Similarly, we have to pay the price of a factor $n$ in the probability, as can be seen from the union bound
    \begin{align*}
        \PP\Bigg\{\sup_{x \in \Theta_n^*}\frac{R_{k(x)}(x)}{\zeta_{h_+}} \geq 1\Bigg\} &\leq \PP\Bigg\{\bigcup_{\ell = \lceil \log n \rceil}^n\Bigg\{\sup_{x \in \RR^d}\frac{R_{\ell}(x)}{\zeta_{h_+}(x)} \geq 1\Bigg\} \Bigg\}\\
        &\leq \sum_{\ell = \lceil \log n \rceil}^n\PP\Bigg\{\sup_{x \in \RR^d}\frac{R_{\ell}(x)}{\zeta_{h_+}(x)}\geq 1\Bigg\}\\
        &\leq n^{1 - \tau}.
    \end{align*}
\end{remark}
\begin{proposition}
    \label{prop.neighbours.distance.lower.bound.zeta}
    Let $\tau > 0$ and $V_\tau := K_0^{1/(2d)}\vee(8d + 4\tau).$ If $k \geq \lceil V_\tau\log n\rceil,$ and $h_- = h_-(k) := k/(2n(1 + \lambda_\tau^-)),$ with 
    \begin{align*}
        \lambda_\tau^- := 1 + \sqrt{\frac 52},
    \end{align*}
    then,
    \begin{align*}
        \PP\Bigg\{\inf_{x \in \RR^d}\frac{R_k(x)}{\zeta_{h_-}(x)} \leq 1\Bigg\} \leq n^{-\tau}.
    \end{align*}
\end{proposition}
We further define
\begin{align}
    \label{eq.def.dtau}
    D_\tau := \frac{1}{2 + \sqrt{10}},
\end{align}
so that $h_- = D_\tau k/n.$

\begin{proof}
    Applying Proposition \ref{prop.Bernstein.general} $(ii)$ with $\delta = n^{-\tau}$ and $h = k/(2(1 + \lambda)n)$ with arbitrary $\lambda \in \RR_+^*,$ yields that for all $n$ satisfying
    \begin{align}
        \label{eq.n.condition.nearest.neighbours.bound.zeta.lower}
        n \geq \frac{8d(1 + \lambda/3)(1 + \lambda)n}{\lambda^2k}\log\Big(K_0^{1/(2d)}n^{1 + \tau/(2d)}k^{-1}\Big),
    \end{align}
    it holds that
    \begin{align}
        \label{eq.bound.zeta.neighbour.lower.hp}
        \PP\Bigg\{\sup_{x \in \RR^d}\frac{1}{n}\sum_{i=1}^n \mathds{1}(X_i \in \B(x, \zeta_{h_-}(x))) \geq \frac kn\Bigg\} \leq n^{-\tau},
    \end{align}
    where $K_0$ is defined in Proposition \ref{prop.Bernstein.general}. The high-probability bound \eqref{eq.bound.zeta.neighbour.lower.hp} implies that
    \begin{align*}
        \PP\Bigg\{\inf_{x \in \RR^d}\frac{R_k(x)}{\zeta_{h_-}(x)} \geq 1\Bigg\} \leq n^{-\tau}.
    \end{align*}
    Taking $\lambda_\tau^- = 1 + \sqrt{5}/2$ in \eqref{eq.n.condition.nearest.neighbours.bound.zeta.lower} leads to
    \begin{align*}
        k \geq 4\log\Big(K_0n^{2d + \tau}k^{-2d}\Big),
    \end{align*}
    and it is easy to check that this inequality is satisfied for all $n \geq 2d$ and $k \geq \lceil V_\tau \log n\rceil.$
\end{proof}
We define the events
\begin{align*}
    L_n^{\P}(k) := \Bigg\{\inf_{x \in \RR^d} \frac{R_k^{\P}(x)}{\zeta^{\P}_{h_+(k)}(x)}\geq 1\Bigg\},\text{ and }L_n^{\P} &:= \bigcap_{k= \lceil V_\tau\log n\rceil}^n L_n^{\P}(k).
\end{align*}
We also define the set
\begin{align*}
    \Delta_n^{\P}(c, k) := \bigg\{x \in \RR^d: p(x) \geq \frac{ck}{n}\bigg\}.
\end{align*}

\begin{lemma}
    \label{lem.bound.neighbours.density}
    Let $\tau > 0$ and $p$ be the density of $\P.$ If $k \geq \lceil\log n\rceil,$ then, for all $n \geq 2d,$
    \begin{align*}
         U_n^{\P}(k) \subseteq \Bigg\{\sup_{x \in \Delta_n^{\P}(c_u, k)}p(x)R_k(x)^d \leq \frac{B_\tau k}{a_- n}\Bigg\},
    \end{align*}
    with $c_u := B_\tau/(a_-r_-^d).$
\end{lemma}

\begin{proof}
    Since $\P \in \mP,$ we can apply \eqref{eq.minimal.mass} to obtain that for all $x \in \RR^d,$
    \begin{align*}
        h_+ = \P\{\B(x, \zeta_{h_+}(x))\} \geq a_-\big(\zeta_{h_+}(x)\wedge r_-\big)^dp(x).
    \end{align*}
    By rearranging the previously displayed equation, we obtain
    \begin{align*}
        \zeta_{h_+}(x)\wedge r_+ \leq \bigg(\frac{h_+}{a_-p(x)}\bigg)^{1/d}.
    \end{align*}
    We observe that if $x \in \Delta_n^{\P}(c_u, k),$ then $\P\{\B(x, r_-)\} \geq a_-r_-^dp(x) \geq B_\tau k/n = h_+.$ Hence, $\zeta_{h_+}(x) \leq r_-$ and for all $x \in \Delta_n^{\P}(c_u, k),$ on $U_n^{\P}(k),$
    \begin{align*}
        R_k^d(x) \leq \frac{B_\tau k}{a_-np(x)},
    \end{align*}
    which implies the result.
\end{proof}

\begin{remark}
    \label{rem.bound.variable.k.density}
    Given a function $k \colon \RR^d \to \{1, \dots, n\}$ and a constant $c>0,$ we define the set
    \begin{align*}
        \Theta_n^{\P}(c, k) := \Bigg\{x \in \RR^d: p(x)\geq \frac{ck(x)}{n}\Bigg\}.
    \end{align*}
    As a direct implication of Lemma \ref{lem.bound.neighbours.density}, we obtain,
    \begin{align*}
        \Bigg\{\sup_{x \in \Theta_n^{\P}(c, k)}\frac{p(x)R_{k(x)}(x)^d}{k(x)} \geq \frac{B_\tau}{a_- n}\Bigg\} &= \bigcup_{i=\lceil\log n\rceil}^n\Bigg\{\sup_{x \in k^{-1}(i)\cap \Delta_n^{\P}(c, i)}p(x)R_i(x)^d \geq \frac{B_\tau i}{a_- n}\Bigg\}\\
        &\subseteq \bigcup_{i=\lceil \log n \rceil}^n\Bigg\{\sup_{x \in \Delta_n^{\P}(c, i)} p(x)R_i(x)^d \geq \frac{B_\tau i}{a_- n}\Bigg\}\\
        &= \bigg(\bigcap_{i=\lceil \log n\rceil}^n U_n^{\P}(i)\bigg)^c = U_n^{\P c}.
    \end{align*}
    Hence,
    \begin{align*}
        \PP\Bigg\{\sup_{x \in \Theta_n^{\P}(c, k)}\frac{p(x)R_{k(x)}(x)^d}{k(x)} \leq \frac{B_\tau}{a_- n}\Bigg\} \geq 1 - n^{1-\tau}.
    \end{align*}
\end{remark}

\begin{lemma}
    \label{lem.lower.bound.neighbours.density}
    Let $\tau > 0.$ If $k \geq \lceil V_\tau\log n\rceil,$
    \begin{align*}
        \Bigg\{\inf_{x \in \Delta_n^{\P}(c_{\ell}, k)}p(x)R_k(x)^d \leq \frac{D_\tau k}{a_+ n}\Bigg\} \subseteq L_n^{\P}(k)^c,
    \end{align*}
    where $c_\ell := D_\tau/(a_-(r_+\wedge r_-)^d).$
\end{lemma}

\begin{proof}
    We work on the event $L_n^{\P}(k).$ For all $x \in \supp(\P)$ such that $\zeta_{h_-}(x) \leq r_+,$ the maximal mass property \eqref{eq.maximal.mass} implies that
    \begin{align*}
        h_- = \P\{\B(x, \zeta_{h_-}(x))\} \leq a_+p(x)\zeta_{h_-}(x)^d.
    \end{align*}
    The latter, once rearranged, yields
    \begin{align*}
        \zeta_{h_-}(x) \geq \bigg(\frac{D_\tau k}{a_+np(x)}\bigg)^{1/d}.
    \end{align*}
    It remains to prove that if $x \in \Delta_n^{\P}(c_{\ell}, k),$ then $\zeta_{h_-}(x) \leq r_+.$ Let $x \in \Delta_n^{\P}(c_\ell, k)$ and assume that $\zeta_{h_-}(x) > r_+.$ The minimal mass property \eqref{eq.minimal.mass} implies that
    \begin{align*}
        h_- = \P\{\B(x, \zeta_{h_-}(x))\} \overset{(i)}{>} \P\{\B(x, r_+)\} \geq a_-(r_-\wedge r_+)^dp(x) \geq \frac{D_\tau k}{n} = h_-,
    \end{align*}
    which is a contradiction. Hence, for all $x \in \Delta_n^{\P}(c_\ell, k),$ it holds that $\zeta_{h_-}(x) \leq r_+,$ and a fortiori, that
    \begin{align*}
        L_n^{\P}(k) \subseteq \Bigg\{\inf_{x \in \Delta_n^{\P}(c_\ell, k)} p(x) R_k(x)^d \geq \frac{D_\tau k}{a_+ n}\Bigg\}.
    \end{align*}
\end{proof}
Let $\wh f$ be a local $k$-NN regression estimator with neighbour function $k \colon \RR^d \to \{1, \dots, n\}$ independent from $Y_i|X_i$ for $1 \leq i \leq n.$ We define the associated pointwise average function 
\begin{align}
    \label{eq.average.knn}
    \ol f\colon x \mapsto \E\big[\wh f(x)|X_1, \dots, X_n\big].
\end{align}

\begin{proposition}\label{prop.variance.bound.general.case}
Let $\wh f$ be as in \eqref{eq.weighted.knn}, with neighbour function $k \colon \RR^d \to \{1, \dots, n\}.$ Let $\ol f$ be defined by \eqref{eq.average.knn}.
Then, for all $n\geq 2d,$ it holds that
\begin{align*}
    \PP\Bigg\{\sup_{x \in \RR^d}k(x)^{1/2}\big|\wh f(x) - \ol f(x)\big| \geq \eta\Bigg\} &\leq K_dn^{2d + 1}e^{-H \eta^2},
\end{align*}
where $K_d = 2(25/d)^d,$ and $H = \lambda_0^2/(8\sigma_0).$
\end{proposition}

\begin{proof}
For any $x \in \RR^d$, we have
\begin{align*}
\wh f(x)-\ol f(x)=\sum_{i=1}^nk(x)^{-1}[Y_i(x)-f(X_i(x))]\mathds{1}(i \leq k(x)).
\end{align*}
Let $Z_i(x) := Y_i(x) - f(X_i(x))$ for all $1 \leq i \leq n$ and all $x \in \RR^d.$ By Proposition 8.1 in \cite{Biau2015Lectures} and the regression model \eqref{eq.model}, conditionally on $X_1,\ldots,X_n$, the random variables $\{Z_i(x)\}_{i=1}^n$ are mutually independent with zero mean. By assumption, they satisfy the uniform noise condition
\begin{align}\label{equ::uniformnoise}
\sup_{x^n\in \RR^{d\times n}} 
\E \Big[ \exp\big(\lambda_0|Z_i(x)|\big) \big| X^n = x^n \Big]
\leq \sigma_0 
< \infty,
\end{align}
where $\lambda_0$ and $c_0$ are some constants, and $x^n \in \RR^{d\times n}.$ 
By Lemma 12.1 in \cite{Biau2015Lectures}, 
for any $x \in \RR^d,$ and any $\eta \leq \min(1, 2\sigma_0) / \lambda_0$, it holds that
\begin{align}
    \label{eq.Chernoff.bound.knn}
    \PP \Big\{\big|\wh f(x) - \ol f(x)\big| \geq \eta \Big| X^n = x^n \Big\}
    \leq 2 \exp \bigg( - \frac{\eta^2 \lambda_0^2k(x)}{8 \sigma_0} \bigg).
\end{align}
To derive a high-probability bound that is uniform over $x \in \RR^d,$ we introduce the following set of permutations
\begin{align*}
    \mW(x^n) := \bigg\{\sigma \in \mS_n: \exists x \in \RR^d,\ X_i(x) = X_{\sigma(i)},\ \forall i \in \{1, \dots, n\}\bigg\},
\end{align*}
where $\mS_n$ is the symmetric group over $\{1, \dots, n\}.$ Theorem 12.2 in \cite{Biau2015Lectures} states that for all $n \geq 2d,$ and all $x^n \in \RR^{d\times n},$ the cardinality of $\mW(x^n)$ is bounded by
\begin{align*}
    \ell(x^n) := \card(\mW(x^n)) \leq \Big(\frac{25}{d}\Big)^dn^{2d}.
\end{align*}
In particular, given $x^n \in \RR^{d\times n},$ we can define the equivalence relationship $\sim_{x^n}$ on $\RR^d$ as, for all $z_1, z_2 \in \RR^d,$
\begin{align*}
    \big(z_1 \sim_{x^n} z_2\big) \Longleftrightarrow \Big(\forall i \in \{1, \dots, n\}: x_i(z_1) = x_i(z_2)\Big),
\end{align*}
where $x_i(z)$ is the $i$-th nearest neighbour of $z$ among $\{x_1, \dots, x_n\}.$ Further, $\RR^d/\sim_{x^n}$ has cardinality at most $\card(\mW(x^n)),$ and is in bijection with $\mW(x^n).$ Let then $\{A_\sigma\}_{\sigma \in \mW(x^n)}$ be the elements of $\RR^d/\sim_{x^n}.$ We can now derive a high probability uniform upper bound by partitioning $\RR^d$ according to $\sim_{x^n},$ using the union bound, and applying \eqref{eq.Chernoff.bound.knn} as follows 
\begin{align*}
    \PP\Bigg\{\sup_{x \in \RR^d}&k(x)^{1/2}\big|\wh f(x) - \ol f(x)\big| \geq \eta\Big| X^n = x^n\Bigg\} \\
    &\leq \PP\Bigg\{\bigcup_{\sigma \in \mW(X^n)}\Bigg\{\sup_{x \in A_\sigma}\Big|\sum_{i=1}^n \frac{k(x)^{1/2}}{k(x)}\big(Y_{\sigma(i)} - f(X_{\sigma(i)})\big)\Big| \geq \eta \bigg| X^n = x^n\Bigg\}\Bigg\}\\
    &\leq \sum_{\sigma \in \mW(x^n)}\PP\Bigg\{\sup_{\ell \in \{1, \dots, n\}}\Big|\sum_{i=1}^\ell \ell^{-1/2}\big(Y_{\sigma(i)} - f(X_{\sigma(i)})\big)\Big| \geq \eta \bigg| X^n = x^n\Bigg\}\\
    &\leq \sum_{\sigma \in \mW(x^n)}\sum_{\ell = 1}^n\PP\Bigg\{\Big|\sum_{i=1}^\ell \ell^{-1/2}\big(Y_{\sigma(i)} - f(X_{\sigma(i)})\big)\Big| \geq \eta \bigg| X^n = x^n\Bigg\}\\
    &\leq \sum_{\sigma \in \mW(x^n)}\sum_{\ell = 1}^n 2\exp\bigg(-\frac{\eta^2\lambda_0^2}{8c_0}\bigg)\\
    &\leq 2\Big(\frac{25}d\Big)^dn^{2d+1} \exp\bigg(-\frac{\eta^2\lambda_0^2}{8c_0}\bigg).
\end{align*}
We finish the proof by taking the expectation with regards to $X_1, \dots, X_n$ on both sides.
\end{proof}

\begin{corollary}
    \label{cor.variance.bound.nice.version}
    Under the conditions of Proposition \ref{prop.variance.bound.general.case}, we have for all $\tau > 0,$ all $n \geq 3 \vee 2d,$
    \begin{align*}
        \PP\Bigg\{\sup_{x \in \RR^d}k(x)^{1/2}\big|\wh f(x) - \ol f(x)\big|\geq \sqrt{C_\tau \log n}\Bigg\} \leq n^{-\tau},
    \end{align*}
    with 
    \begin{align*}
        C_\tau := \frac{\log K_d}{H}\bigg(1 + \frac{2d + \tau + 1 }{\log(K_d))}\bigg)
    \end{align*} where $H$ and $K_d$ are defined in Proposition \ref{prop.variance.bound.general.case}.
\end{corollary}

\begin{proof}
    Once rearranged, the conclusion of Proposition \ref{prop.variance.bound.general.case} reads
    \begin{align*}
        \PP\Bigg\{\forall x\in\RR^d, \ \big|\wh f(x) - \ol f(x)\big| \leq \sqrt{\frac{1}{Hk(x)}\log\Big(\frac{K_dn^{2d + 1}}{\delta}\Big)}\Bigg\} \geq 1 - \delta.
    \end{align*}
    We now apply this result with 
    \begin{align*}
        \eta_- \geq \sqrt{\frac{\log K_d}H\bigg(1 + \frac{2d + \tau + 1}{\log K_d}\log n\bigg)}
    \end{align*}
    to obtain
    \begin{align*}
        \PP\Bigg\{\sup_{x \in \RR^d}k(x)^{1/2}|\wh f(x) - \ol f(x)| \geq \eta_-\Bigg\} \leq n^{-\tau}.
    \end{align*}
    Finally, since $n\geq 3, \ \log n \geq 1$ and we can pick
    \begin{align*}
        \eta = \sqrt{C_\tau\log n} = \sqrt{\frac{\log K_d}{H}\bigg(1 + \frac{2d + \tau + 1}{\log K_d}\bigg)\log n} \geq \eta_-,
    \end{align*}
    and obtain the claimed result.
\end{proof}
We define the event
\begin{align*}
    V_n^{\P}(\wh f) := \Bigg\{\sup_{x \in \RR^d}k(x)^{1/2}|=\big|\wh f(x) - \ol f(x)\big| \leq \sqrt{C_\tau\log n}\Bigg\}.
\end{align*}

\begin{lemma}
    \label{lem.bias.term}
    If $f \in \mH_{\beta}(F, L),$ then, for all $x \in \RR^d,$
        \begin{align*}
            |\ol f(x) - f(x)| \leq 2F \wedge L \sum_{i=1}^{k(x)} k(x)^{-1} R_i(x)^\beta,
        \end{align*}
        Additionally, if $\P \in \mP,$ and $k(x) \geq \lceil \log n\rceil$ for all $x \in \RR^d,$ then
        \begin{align*}
            \bigcap_{i\in k(\RR^d)}U_n^{\P}(i) \subseteq \Bigg\{\sup_{x \in \Theta_n^{\P}(c_u, k)}\frac{p(x)^{\beta/d}|\ol f(x) - f(x)|}{k(x)^{\beta/d}} \leq \bigg(\frac{B_\tau}{a_- n}\bigg)^{\beta/d}\Bigg\}.
        \end{align*}
\end{lemma}
\begin{proof}
    Since $f \in \mH_{\beta}(F, L),$ we get for all $x \in \RR^d,$ and all $i \in \{1, \dots, k(x)\},$
    \begin{align*}
        \big|f(X_i(x)) - f(x)\big| &\leq \big(f(X_i(x)) + f(x)\big) \wedge L \|X_i(x) - x\|^{\beta}
        \nonumber\\
        &\leq 2F \wedge L R_i(x)^\beta.
    \end{align*}
    This together with the definition of $\ol f(x)$ and $\sum_{i=1}^{k(x)} k(x)^{-1} = 1$ yields 
    \begin{align*}
        \big|\ol f(x) - f(x) \big|
        &= \bigg|\sum_{i=1}^{k(x)} k(x)^{-1} f(X_i(x)) - f(x) \bigg|\\
        &\leq \sum_{i=1}^{k(x)} k(x)^{-1} \big|f(X_i(x)) - f(x) \big|\\
        &\leq \sum_{i=1}^{k(x)} k(x)^{-1}\big(2F \wedge L R_i(x)^\beta\big)\\
        &\leq 2F \wedge L \sum_{i=1}^{k(x)} k(x)^{-1} R_i(x)^\beta.
    \end{align*}
    
    This proves the first statement. The second statement is obtained by further upper bounding the previously derived inequality by $|\ol f(x) - f(x)| \leq LR_k(x)^\beta$ and applying Lemma \ref{lem.bound.neighbours.density}.
\end{proof}

\begin{proposition}[Bias-Variance decomposition]
    \label{prop.decomp.knn}
    Under the assumptions of model \eqref{eq.model}, if $\wh f$ is a local $k$-NN with neighbour function $k\colon \RR^d \to \{\lceil \log n\rceil, \dots, n\},$ and if $\P \in \mP,$ then, for all $n \geq 2d \vee 3$ and all $\Delta \subseteq \Theta_n^{\P}(c_u, k),$
    \begin{multline*}
        \Bigg\{\forall x \in \RR^d,\ \big(\wh f(x) - f(x)\big)^2 \leq 2C_\tau\frac{\log n}{k(x)} + 2L\bigg(\frac{B_\tau k(x)}{a_-np(x)}\bigg)^{2\beta/d}\mathds{1}(x \in \Delta) + 8F^2\mathds{1}(x \in \Delta^c)\Bigg\}\\
        \supseteq \bigg(\bigcap_{i \in k(\RR^d)}U_n^{\P}(i)\bigg)\bigcap V_n^{\P}(\wh f).
    \end{multline*}
\end{proposition}

\begin{proof}
    Lemmas \ref{lem.bias.term} and \ref{cor.variance.bound.nice.version} imply that, on the event
    \begin{align*}
        \bigg(\bigcap_{i \in k(\RR^d)} U_n^{\P}(i)\bigg)\bigcap V_n^{\P}(\wh f),
    \end{align*}
    for all $x \in \RR^d,$
    \begin{align*}
        \big(\wh f(x) - f(x)\big)^2 &\leq 2\big|\wh f(x) - \ol f(x)\big|^2 + 2|\ol f(x) - f(x)\big|^2\\
        &\leq \frac{2C_\tau \log n}{k(x)} + 2L\bigg(\frac{B_\tau k(x)}{a_-np(x)}\bigg)^{2\beta/d}\mathds{1}(x \in \Delta) + 8F^2\mathds{1}(x \in \Delta^c).
    \end{align*}
    This finishes the proof.
\end{proof}

\section{Proofs of the main results}

Let $\gamma > 0, \P_{\sX} \in \mP,$ and $\Q_{\sX} \in \mQ(\P_{\sX}, \gamma, \rho)$. Let $X \sim \P_{\sX},\ X' \sim \Q_{\sX},\ Y = f(X) + \eps,$ and $Y' \sim f(X) + \eps',$ where $\eps, \eps'$ are two i.i.d.\ random variables that are independent from $X$ and $X'$ and satisfy the uniform noise condition \eqref{eq.uniform.noise}. Let $(X_1, Y_1), \dots, (X_n, Y_n)$ be $n$ i.i.d.\ copies of $(X, Y)$ and $(X_1', Y_1'), \dots, (X_m', Y_m')$ be $m$ i.i.d.\ copies of $(X', Y').$ For $k \in \NN$ and $x \in \RR^d,$ whenever those quantities make sense, we denote by $R_k^{\P}(x) := \|x - X_k(x)\|$ and by $R_k^{\Q} := \|x - X_k'(x)\|.$

\subsection{Transfer without target knowledge}

This section considers the case where we construct an estimator based solely on the data $\{(X_i, Y_i)\}_{i=1}^n.$ Our first result concerns the convergence rates of a standard $k$-NN estimator under \eqref{eq.covariate.shift}.
\begin{theorem}[Rates for standard estimator]
    \label{th.rates.standard.knn.proof}
    Denote $u := 2\beta/d \wedge \gamma,\ r = 2\beta/(2\beta + d) \wedge \gamma/(\gamma + 1).$ Let $\k > 0$ and $\tau > 0.$ If $\wh f$ is defined as in \eqref{eq.weighted.knn} with neighbour function set to be a constant $k$ defined as
    \begin{align*}
        k = \big\lceil \k \log(n)^{1 - r}n^{r}\big\rceil,
    \end{align*} 
    then, there exists $N = N(\gamma, \beta, d, \k) \geq 2d \vee 3$ such that for all $n \geq N,$ with probability at least $1 - 2n^{-\tau}$, it holds that
    \begin{align*}
        \mE_{\Q_{\sX}}(\wh f, f) &\leq \Bigg[\frac{2C_\tau}{\kappa} + \bigg[2Lr_-^{2\beta} + 8F^2\bigg]\bigg(\frac{\kappa B_\tau}{a_-r_-^d}\bigg)^{u}\mT_{\P_{\sX}}(u)\Bigg]\bigg(\frac{\log n}{n}\bigg)^{r}.
    \end{align*}
\end{theorem}
\begin{remark}
    The upper bound in Theorem \ref{th.rates.standard.knn.proof} is optimized by taking $\k$ to be
    \begin{align*}
        \k^* := \bigg(\frac{2C_\tau}{u\mT_{\P_{\sX}}(u)\big(2Lr_-^{2\beta} + 8F^2\big)}\bigg)^{1 - r}\bigg(\frac{a_-r_-^d}{B_\tau}\bigg)^{r}.
    \end{align*}
    This leads to the high probability rate
    \begin{align*}
        \mE_{\Q_{\sX}}(\wh f, f) \leq \bigg(\frac{2C_\tau B_\tau}{a_-r_-^d}\bigg)^{r}\Big((2Lr_-^{2\beta} + 8F^2)u\mT_{\P_{\sX}}(u)\Big)^{1-r}(1 + u^{-1})\bigg(\frac{\log n}n\bigg)^{r}.
    \end{align*}
\end{remark}
\begin{remark}
    A careful reading of the proof of Theorem \ref{th.rates.standard.knn.proof} shows that taking
    \begin{align*}
        k = \big\lceil \k\log(n + t)^{1 - r}n^r\rceil,
    \end{align*}
    with $t > 0,$ allows to obtain a rate of $(\log (n + t)/n)^r$ instead of $(\log n/n)^r.$ This will be useful for the proofs of the rates of the two-sample estimators.
\end{remark}

\begin{proof}
    Let $N \geq 2d\vee 3$ be such that $k \geq \log(N)\wedge 2.$ It can be checked that $N = N(\gamma, \beta, d, \k).$ We work on the event
    \begin{align}
        \label{eq.event.standard.one.sample.knn}
        U_n^{\P_{\sX}}(k)\cap V_n^{\P_{\sX}}(\wh f).
    \end{align}
    Let $\Delta:= \Delta_n^{\P_{\sX}}(c_u, k).$ Proposition \ref{prop.decomp.knn} shows that, on this event, it holds that
    \begin{align}
        \label{eq.excess.risk.standard.knn}
        \mE_{\Q}(\wh f, f)
        & \leq \int \frac{2C_\tau\log n}{k} + 2L\bigg(\frac{B_\tau k}{a_-np(x)}\bigg)^{2\beta/d}\mathds{1}(x \in \Delta) + 8F^2\mathds{1}(x \in \Delta) \, d\Q_{\sX}(x) \nonumber\\
        & = \frac{2C_\tau\log n}{k} + \int_{\Delta}2L\bigg(\frac{B_\tau k}{a_-np(x)}\bigg)^{2\beta/d}\, d\Q_{\sX}(x) + 8F^2\Q_{\sX}\{\Delta^c\}
    \end{align}
     We then use Lemma \ref{lem.ratio.bound} to bound the third term in \eqref{eq.excess.risk.standard.knn} as follows
    \begin{align}
        \label{eq.standard.knn.one.sample.1}
        8F^2\Q_{\sX}\{\Delta^{c}\} &= 8F^2\int_{\Delta^{c}}q(x)\, dx\nonumber\\
        &\leq 8F^2\int_{\Delta^{c}} \frac{q(x)}{p(x)^u} p(x)^{u}\, dx\nonumber\\
        &\leq 8F^2\bigg(\frac{B_\tau k}{a_- n r_-^d}\bigg)^{u}\mT_{\P_{\sX}}(u).
    \end{align}
    The second term of \eqref{eq.excess.risk.standard.knn} is bounded as follows
    \begin{align}
        \label{eq.standard.knn.one.sample.2}
        \int_{\Delta}2L\bigg(\frac{B_\tau k}{a_-np(x)}\bigg)^{2\beta/d}\, d\Q_{\sX}(x) &\leq 2L\bigg(\frac{B_\tau k}{a_- n}\bigg)^{2\beta/d}\int_{\Delta}\frac{q(x)}{p(x)^u}p(x)^{u - 2\beta/d}\, dx\nonumber\\
        &\leq 2L\mT_{\P_{\sX}}(u)r_-^{2\beta}\bigg(\frac{B_\tau k}{a_-n r_-^d}\bigg)^{u}\nonumber\\
        &\leq 2L\mT_{\P_{\sX}}(u)r_-^{2\beta - ud}\bigg(\frac{B_\tau k}{a_-n}\bigg)^{u}\nonumber\\
        &\overset{(i)}{\leq} 2L\mT_{\P_{\sX}}(u)r_-^{2\beta - ud}\Bigg(\frac{B_\tau \k}{a_-}\bigg(\frac{\log n}{n}\bigg)^{1 - r}\bigg)^{u}\nonumber\\
        &\overset{(ii)}{\leq} 2L\mT_{\P_{\sX}}(u)r_-^{2\beta - ud}\bigg(\frac{2B_\tau \k}{a_-}\bigg)^u\bigg(\frac{\log n}{n}\bigg)^{r}\nonumber\\
        &\leq 2L\mT_{\P_{\sX}}(u)r_-^{2\beta - ud}\bigg(\frac{2B_\tau \k}{a_-}\bigg)^u\bigg(\frac{\log n}{n}\bigg)^{r},
    \end{align}
    where step $(i)$ involves plugging in our choice of $k,$ step $(ii)$ uses $\lceil x\rceil \leq 2x$ if $x \geq 1$ and the fact that $u(1 - r) = r.$ Plugging \eqref{eq.standard.knn.one.sample.1}, \eqref{eq.standard.knn.one.sample.2}, and our choice of $k$ in \eqref{eq.excess.risk.standard.knn} leads to obtain, on the event \eqref{eq.event.standard.one.sample.knn},
    \begin{align*}
         \mE_{\Q}(\wh f, f) &\leq \Bigg[\frac{2C_\tau}{\kappa} + \bigg[2Lr_-^{2\beta} + 8F^2\bigg]\bigg(\frac{\kappa B_\tau}{a_-r_-^d}\bigg)^{u}\mT_{\P_{\sX}}(u)\Bigg]\bigg(\frac{\log n}{n}\bigg)^{r}.
    \end{align*}
    We conclude the proof by using Proposition \ref{prop.neighbours.distance.bound.zeta}, Corollary \ref{cor.variance.bound.nice.version} and the union bound to obtain
    \begin{align*}
        \PP\Big\{U_n^{\P_{\sX}}(k) \cap V_n^{\P_{\sX}}(\wh f)\Big\} \geq 1 - 2n^{-\tau}.
    \end{align*}
\end{proof}
Our next result concerns the local $k$-NN estimator. Our choice of neighbour function $k$ involves a density estimation step. The definition of a $\ell$-nearest neighbour density estimator $\wh p$ is given by \eqref{eq.density.estimator}. We begin with a preliminary result which shows that $\wh p \asymp p$ on a high-density region.

\begin{lemma}
    \label{lem.density.ratio}
    Let $\P \in \mP$ such that $\P$ satisfies \eqref{eq.maximal.mass} with constants $a_+, r_+.$ Denote by $p$ the density of $\P$, and let $\wh p$ be the $\ell$-NN density estimator defined above.
    \begin{enumerate}
        \item [$(i)$] If $\ell \geq \lceil V_\tau\log n\rceil,$ then
        \begin{align*}
            L_n^{\P}(\ell) \subseteq \Bigg\{\sup_{x \in \Delta_n^{\P}(c_\ell, \ell)} \frac{\wh p(x)}{p(x)} \leq \frac{a_+}{D_\tau}\Bigg\}.
        \end{align*}
        \item [$(ii)$] If $\ell \geq \lceil \log n \rceil,$ then
        \begin{align*}
            U_n^{\P}(\ell) \subseteq \Bigg\{\inf_{x \in \Delta_n^{\P}(c_u, \ell)} \frac{\wh p(x)}{p(x)} \geq \frac{a_-}{B_\tau}\Bigg\}
        \end{align*}
    \end{enumerate}
\end{lemma}

\begin{proof}
    $(i)$ We work on the event $L_n^{\P}(\ell).$ For all $x \in \Delta_n^{\P}(c_\ell, \ell)$ we have
    \begin{align*}
        \frac{\wh p(x)}{p(x)} = \frac{\ell}{np(x)R_{\ell}(x)^d} \leq \frac{a_+}{D_\tau},
    \end{align*}
    where the last inequality comes from Lemma \ref{lem.lower.bound.neighbours.density}. This shows that 
    \begin{align*}
        L_n^{\P}(\ell) \subseteq \Bigg\{\sup_{x \in \Delta_n^{\P}(c_\ell, \ell)} \frac{\wh p(x)}{p(x)} \leq \frac{a_+}{D_\tau}\Bigg\}.
    \end{align*}
    $(ii)$ We work on the event $U_n^{\P}(\ell).$ For all $x \in \Delta_n^{\P}(c_u, \ell),$ we have
    \begin{align*}
        \frac{\wh p(x)}{p(x)} = \frac{\ell}{np(x)R_\ell(x)^d} \geq \frac{a_-}{B_\tau},
    \end{align*}
    where the last inequality comes from Lemma \ref{lem.bound.neighbours.density}. This shows that 
    \begin{align*}
        U_n^{\P}(\ell) \subseteq \Bigg\{\inf_{x \in \Delta_n^{\P}(c_\ell, \ell)} \frac{\wh p(x)}{p(x)} \geq \frac{a_-}{B_\tau}\Bigg\}.
    \end{align*}
\end{proof}

\begin{remark}
    This shows that, if $\ell \geq \lceil V_\tau\log n\rceil,$ then, on $L_n^{\P}(\ell) \cap U_n^{\P}(\ell),$ for all $x \in \Delta_n^{\P}(c_u\vee c_\ell, \ell),$
    \begin{align*}
        \frac{a_-}{B_\tau} \leq \frac{\wh p(x)}{p(x)}\leq \frac{a_+}{D_\tau}.
    \end{align*}
\end{remark}
\begin{lemma}
    \label{lem.inclusions}
    Let $\P \in \mP,\ \tau > 0$ and $\k > 0.$ Consider the function
    \begin{align*}
        k(x) = n \wedge \Big\lceil\k\log(n)^{d/(2\beta + d)}\big(n\wh p(x)\big)^{2\beta/(2\beta + d)} \Big\rceil \vee \lceil \log n\rceil, 
    \end{align*}
    where $\wh p$ is a $\ell$-NN density estimator for $\P,$ with $\ell \geq \lceil V_\tau\log n\rceil.$ Let $N = N(\k, \tau, \P)$ satisfy
    \begin{align*}
        \frac N{\log N} = \k^{(2\beta + d)/d}\bigg(\frac{a_+\|p\|_\infty}{D_\tau}\bigg)^{2\beta/d},
    \end{align*}
    and define the constant
    \begin{align*}
        C := c_u\vee c_\ell \vee \frac{B_\tau}{a_-\k^{(2\beta + d)/(2\beta)}V_\tau} \vee \frac{(2c_u\k)^{(2\beta + d)/d}}{V_\tau}\bigg(\frac{a_+}{D_\tau}\bigg)^{2\beta/d}.
    \end{align*}
    Then, the following assertions hold
    \begin{enumerate}
        \item [$(i)$] If $n \geq N,$ then 
            \begin{align*}
                L_n^{\P}(\ell)\cap U_n^{\P}(\ell) \subseteq \Bigg\{\forall x \in \Delta_n^{\P}(C, \ell), \ k(x) = \Big\lceil\k\log(n)^{d/(2\beta + d)}\big(n\wh p(x)\big)^{2\beta/(2\beta + d)}\Big\rceil\Bigg\}.
            \end{align*}
        \item [$(ii)$] If $n \geq N\vee e,$ then
            \begin{align*}
                 L_n^{\P}(\ell)\cap U_n^{\P}(\ell) \subseteq \Bigg\{\Delta_n^{\P}(C, \ell) \subseteq \Theta_n^{\P}(c_u, k)\Bigg\}.
            \end{align*}
    \end{enumerate}
\end{lemma}

\begin{proof}
    Work on $L_n^{\P}(\ell)\cap U_n^{\P}(\ell).$ We obtain $(i)$ by showing that for all $x \in \Delta_n^{\P}(C, \ell),$
    \begin{align*}
        \log n \leq \k\log(n)^{d/(2\beta + d)}\big(n\wh p(x)\big)^{2\beta/(2\beta + d)} \leq n.
    \end{align*}
    Let $x \in \Delta_n^{\P}(C, \ell),$ and $u(x) := \k\log(n)^{d/(2\beta + d)}(n\wh p(x))^{2\beta/(2\beta + d)}.$ We begin by applying Lemma \ref{lem.density.ratio} and obtain, since $\Delta_n^{\P}(C, \ell) \subseteq \Delta_n^{\P}(c_u \vee c_\ell, \ell),$
    \begin{align}
        \label{eq.double.ineq}
        \k\log(n)^{d/(2\beta + d)}\bigg(\frac{na_-p(x)}{B_\tau}\bigg)^{2\beta/(2\beta + d)} \leq u(x) \leq \k\log(n)^{d/(2\beta + d)}\bigg(\frac{a_+n\|p\|_\infty}{D_\tau}\bigg)^{2\beta/(2\beta + d)}.
    \end{align}
    It is immediate to check that the leftmost term in \eqref{eq.double.ineq} is lower bounded by $\log(n)$ if
    \begin{align*}
        p(x) \geq \frac{B_\tau}{a_-\k^{(2\beta + d)/(2\beta)}}\frac{\log n}{n},
    \end{align*}
    which is true a fortiori, for all $x \in \Delta_n^{\P}(C, \ell)$ since
    \begin{align*}
        p(x) \geq \frac{C\ell}{n} \geq \frac{B_\tau}{a_-\k^{(2\beta + d)/(2\beta)}V_\tau}\frac{\ell}{n} \geq \frac{B_\tau}{a_-\k^{(2\beta + d)/(2\beta)}}\frac{\log n}{n}.
    \end{align*}
    The rightmost term in \eqref{eq.double.ineq} is upper bounded by $n$ if
    \begin{align*}
        \frac n{\log n} \geq \k^{(2\beta + d)/d}\bigg(\frac{a_+\|p\|_\infty}{D_\tau}\bigg)^{2\beta/d},
    \end{align*}
    which is true by assumption. We have proven the two previous facts for arbitrary $x \in \Delta_n^{\P}(C, \ell),$ therefore,
    \begin{align*}
        L_n^{\P}(\ell)\cap U_n^{\P}(\ell) \subseteq \Bigg\{\forall x \in \Delta_n^{\P}(C, \ell), \ k(x) = \Big\lceil\k\log(n)^{d/(2\beta + d)}\big(n\wh p(x)\big)^{2\beta/(2\beta + d)}\Big\rceil\Bigg\}.
    \end{align*}
    We now prove $(ii).$ Let $x \in \Delta_n^{\P}(C, \ell).$ For $n \geq e \vee N, \ (i)$ implies that $k(x) \geq \lceil \log n\rceil \geq 1.$ Hence, 
\begin{align*}
    \frac {c_u k(x)}n &\leq \frac{2c_u\k\log(n)^{d/(2\beta + d)}(n\wh p(x))^{2\beta/(2\beta + d)}}{n}\\
    &\leq 2c_u\k \bigg(\frac{\log n}{n}\bigg)^{d/(2\beta + d)}\bigg(\frac{a_+p(x)}{D_\tau}\bigg)^{2\beta/(2\beta + d)},
\end{align*}
and the latter is upper-bounded by $p(x)$ since
\begin{align*}
    p(x) \geq \frac{C\ell}{n} \geq \frac{(2c_u\k)^{(2\beta + d)/d}}{V_\tau}\bigg(\frac{a_+}{D_\tau}\bigg)^{2\beta/d}\frac{\ell}{n} \geq (2c_u\k)^{(2\beta + d)/d}\bigg(\frac{a_+}{D_\tau}\bigg)^{2\beta/d}\frac{\log n}{n}.
\end{align*}
We have proven that
\begin{align*}
    p(x) \geq \frac {C\ell}{n} \Rightarrow p(x) \geq \frac{c_uk(x)}{n},
\end{align*}
and, since $x \in \Delta_n^{\P}(C, \ell)$ was arbitrary,
\begin{align*}
    L_n^{\P}(\ell)\cap U_n^{\P}(\ell) \subseteq \Bigg\{\Delta_n^{\P}(C, \ell) \subseteq \Theta_n^{\P}(c_u, k)\Bigg\}.
\end{align*}
This concludes the proof.
\end{proof}

\begin{theorem}
    \label{th.transfer.local.knn.proof}
    Let $\gamma > 0, \ \kappa > 0.$ Let $C$ be the constant defined in Lemma \ref{lem.inclusions}, $r := \gamma\wedge 2\beta/(2\beta + d),$ and $\mT(r) := \int q(x)/p(x)^r\, dx.$ If $\wh p$ is a $\ell$-NN density estimator for $p$ with $\ell \geq \lceil V_\tau \log n\rceil,$ and $\wh f$ is a local $k$-NN estimator defined in \eqref{eq.weighted.knn} with neighbour function $k \colon \RR^d \to \{\lceil \log n\rceil, \dots, n\}$ given by
    \begin{align*}
        k(x) := n\wedge \Big\lceil \kappa\log(n)^{d/(2\beta + d)}(n\wh p(x))^{2\beta/(2\beta + d)}\Big\rceil \vee \big\lceil \log n \big\rceil.
    \end{align*}
    then, there exists an $N > 2d\vee e$ such that for all $n \geq N,$ with probability at least $1 - 3n^{1- \tau},$
    \begin{multline*}
        \mE_{\Q}(f^*, f) \leq \bigg(\frac{V_\tau}{C}\bigg)^{r - 2\beta/(2\beta +d)}\Bigg[\frac{2C_\tau}{\kappa}\bigg(\frac{B_\tau}{a_-}\bigg)^{2\beta/(2\beta + d)} + A\Bigg]\mT_{\P_{\sX}}(r)\bigg(\frac{\log n}{n}\bigg)^{r}\\
        + \Bigg[2C_\tau + 8F^2\Bigg]\mT_{\P_{\sX}}(r)\bigg(\frac{C\ell}{n}\bigg)^{r}.
    \end{multline*}
    with
    \begin{align*}
        A := 2L\bigg(\frac{2\k B_\tau}{a_-}\bigg)^{2\beta/d}\bigg(\frac{a_+}{D_\tau}\bigg)^{4\beta^2/(d(2\beta + d))}.
    \end{align*}
\end{theorem}

\begin{remark}
    In particular, the choice $\ell = \big\lceil V_\tau \log n\big\rceil$ in Theorem \ref{th.transfer.local.knn.proof} leads to
    \begin{align*}
        \mE_{\Q}(\wh f, f) \leq V_\tau^r\mT_{\P_{\sX}}(r)\Bigg[\frac{2C_\tau}{\kappa}\bigg(\frac{B_\tau C}{V_\tau a_-}\bigg)^{r_0} + A\bigg(\frac{C}{C_\tau}\bigg)^{r_0} + 2^r\bigg[2C_\tau + 8F^2\bigg]\Bigg]\bigg(\frac{\log n}n\bigg)^{r},
    \end{align*}
    with probability at least $1 - 3n^{1-\tau},$ where $r_0 = 2\beta/(2\beta + d).$
\end{remark}

\begin{proof}
    Let $X \sim \Q_{\sX}, \ C$ be the constant defined in Lemma \ref{lem.inclusions}, and denote $\Delta = \Delta_n^{\P_{\sX}}(C, \ell).$ We work on the event 
    \begin{align}
        \label{eq.event.local.knn}
        L_n^{\P_{\sX}}\cap U_n^{\P_{\sX}}(\ell)\cap V_n^{\P_{\sX}}(\wh f).
    \end{align}
    By Lemma \ref{lem.inclusions} $(ii),$ $\Delta \subseteq \Theta_n^{\P_{\sX}}(c_u, k),$ and, by definition, $k(x) \geq \lceil \log n\rceil$. Hence, we can apply Proposition \ref{prop.decomp.knn} to obtain
    \begin{align}
        \label{eq.local.k.proof.decomp}
        \mE_{\Q}(\wh f, f) &\leq \E\Bigg[\frac{2C_\tau\log n}{k(X)} + 2L\bigg(\frac{B_\tau k(X)}{a_-np(X)}\bigg)^{2\beta/d}\mathds{1}(X \in \Delta) + 8F^2\mathds{1}(X \in \Delta^c)\Bigg].
    \end{align}
    We first derive pointwise bounds on the three quantities occurring inside the expectation operator in \eqref{eq.local.k.proof.decomp}. For the first term, using the definition of $k,$ Lemma \ref{lem.inclusions}, and Lemma \ref{lem.density.ratio} $(ii)$, we get
    \begin{align}
        \label{eq.local.one.sample.pointwise.variance.bound}
        \frac{2C_\tau\log n}{k(x)} &\leq  \frac{2C_\tau\log n}{\kappa \log(n)^{d/(2\beta + d)} (n \wh p(x))^{2\beta/(2\beta + d)}}\mathds{1}(x \in \Delta) +  \frac{2C_\tau\log n}{\lceil \log n\rceil}\mathds{1}(x \in \Delta^c)\nonumber\\
        &\leq \frac{2C_\tau}{\k}\bigg(\frac{B_\tau \log n}{a_-np(x)}\bigg)^{2\beta/(2\beta + d)}\mathds{1}(x \in \Delta) + 2C_\tau\mathds{1}(x \in \Delta^c).
    \end{align}
    Using Lemma \ref{lem.inclusions} $(i)$ and the condition $n \geq N\vee e,$ we have $k(x) \geq \lceil \log n\rceil \geq 1.$ This allows to use $\lceil a \rceil \leq 2a$ with $a$ being the expression inside the ceiling operator in the expression of $k(x).$ Additionally, we use Lemma \ref{lem.density.ratio} $(i)$ to bound the second term in \eqref{eq.local.k.proof.decomp} as follows
    \begin{align}
        \label{eq.local.one.sample.pointwise.bias.hd.bound}
        2L\bigg(\frac{B_\tau k(x)}{a_-np(x)}\bigg)^{2\beta/d} & \leq 2L\bigg(\frac{B_\tau}{a_-}\bigg)^{2\beta/d}\bigg(\frac{2\k\log(n)^{d/(2\beta + d)}(n\wh p(x))^{2\beta/ (2\beta + d)}}{np(x)}\bigg)^{2\beta/d} \nonumber\\
        &\leq 2L\bigg(\frac{B_\tau}{a_-}\bigg)^{2\beta/d}\Bigg(2\k\bigg(\frac{a_+}{D_\tau}\bigg)^{2\beta/(2\beta + d)}\bigg(\frac{\log(n)}{np(x)}\bigg)^{d/(2\beta + d)}\Bigg)^{2\beta/d} \nonumber\\
        &\leq \underbrace{2L\bigg(\frac{2\k B_\tau}{a_-}\bigg)^{2\beta/d}\bigg(\frac{a_+}{D_\tau}\bigg)^{4\beta^2/(d(2\beta + d))}}_{= A}\bigg(\frac{\log(n)}{np(x)}\bigg)^{2\beta/(2\beta + d)}.
    \end{align}
    We now integrate the previously derived bounds with respect to $\Q_{\sX}.$ We begin by integrating the first term of \eqref{eq.local.one.sample.pointwise.variance.bound}. This gives,
    \begin{align*}
        \int_{\Delta}\frac{2C_\tau \log n}{k(X)}\, d\Q_{\sX}(x) &\leq \frac{2C_\tau}{\kappa}\bigg(\frac{B_\tau\log n}{a_-n}\bigg)^{2\beta/(2\beta + d)}\int_{\Delta}\frac{q(x)}{p(x)^{r}}p(x)^{r - 2\beta/(2\beta + d)}\, dx\\
        &\leq \frac{2C_\tau}{\kappa}\bigg(\frac{B_\tau}{a_-}\bigg)^{2\beta/(2\beta + d)}\mT_{\P_{\sX}}(r)\bigg(\frac{\log n}{n}\bigg)^{2\beta/(2\beta + d)}\bigg(\frac{C\ell}{n}\bigg)^{r - 2\beta/(2\beta + d)}\\
        &\overset{(i)}{\leq} \frac{2C_\tau}{\kappa}\bigg(\frac{B_\tau}{a_-}\bigg)^{2\beta/(2\beta + d)}\mT_{\P_{\sX}}(r)\bigg(\frac{V_\tau}{C}\bigg)^{r - 2\beta/(2\beta + d)}\bigg(\frac{\log n}{n}\bigg)^{r},
    \end{align*}
    where step $(i)$ follows from the fact that $r - 2\beta/(2\beta + d) \leq 0.$ For the second term in \eqref{eq.local.one.sample.pointwise.variance.bound}, we obtain
    \begin{align*}
        \int_{\Delta^c}\frac{2C_\tau\log n}{k(x)}\, d\Q_{\sX}(x) &\leq 2C_\tau\mT_{\P_{\sX}}(r)\bigg(\frac{C\ell}{n}\bigg)^r.
    \end{align*}
    Integrating \eqref{eq.local.one.sample.pointwise.bias.hd.bound} with respect to $\Q_{\sX}$ leads to 
    \begin{align*}
        \int_{\Delta}2L\bigg(\frac{B_\tau k(x)}{anp(x)}\bigg)^{2\beta/d}\, d\Q_{\sX}(x) &\leq A\bigg(\frac{\log n}n\bigg)^{2\beta/(2\beta + d)}\int_{\Delta}\frac{q(x)}{p(x)^{r}}p(x)^{r - 2\beta/(2\beta + d)}\, dx\\
        &\leq A\mT_{\P_{\sX}}(r)\bigg(\frac{\log n}n\bigg)^{2\beta/(2\beta + d)}\bigg(\frac{C\ell}n\bigg)^{r - 2\beta/(2\beta + d)}\\
        &\leq A\mT_{\P_{\sX}}(r)\bigg(\frac{V_\tau}{C}\bigg)^{r - 2\beta/(2\beta + d)}\bigg(\frac{\log n}n\bigg)^{r}
    \end{align*}
    We now bound the last term in \eqref{eq.local.k.proof.decomp}. We use the fact that on $\Delta^c, \ p(x) < C\ell/n$ to obtain
    \begin{align}
        \label{eq.local.one.sample.bias.lp.bound}
        8F^2\Q_{\sX}\{\Delta^c\} = 8F^2\int_{\Delta^c}\frac{q(x)}{p(x)^r}p(x)^{r}\, dx \leq 8F^2\bigg(\frac{C\ell}n\bigg)^r\mT_{\P_{\sX}}(r).
    \end{align}
    Putting everything together yields
    \begin{multline*}
        \mE_{\Q}(f^*, f) \leq \bigg(\frac{V_\tau}{C}\bigg)^{r - 2\beta/(2\beta +d)}\Bigg[\frac{2C_\tau}{\kappa}\bigg(\frac{B_\tau}{a_-}\bigg)^{2\beta/(2\beta + d)} + A\Bigg]\mT_{\P_{\sX}}(r)\bigg(\frac{\log n}{n}\bigg)^{r}\\
        + \Bigg[2C_\tau + 8F^2\Bigg]\mT_{\P_{\sX}}(r)\bigg(\frac{C\ell}{n}\bigg)^{r}.
    \end{multline*}
    Finally, Proposition \ref{prop.neighbours.distance.bound.zeta}, Proposition \ref{prop.neighbours.distance.lower.bound.zeta}, Corollary \ref{cor.variance.bound.nice.version}, and Remark \ref{rem.bound.variable.k} together with the union bound ensure that
    \begin{align*}
        \PP\Bigg\{L_n^{\P_{\sX}}(\ell)\cap U_n^{\P_{\sX}}(\ell)\cap V_n^{\P_{\sX}}(\wh f)\Bigg\} \geq 1 - 2n^{-\tau} - n^{1-\tau} \geq 1 - 3n^{1-\tau}.
    \end{align*}
\end{proof}

\subsection{Transfer with knowledge of target distribution}

We now consider a two-sample local $(k_{\P}, k_{\Q})$-NN estimator with neighbour functions $k_{\P}$ and $k_{\Q}$ defined as in \eqref{eq.two.sample.local.knn}.
Let now $\wh f_{k_{\P}}$ be a one sample $k_{\P}$-NN estimator and $\wh f_{k_{\Q}}$ be a one sample $k_{\Q}$-NN estimator. Let $w_{\P} := k_{\P}/(k_{\P} + k_{\Q})$ and $w_{\Q} := 1 - w_{\P}.$ The expression of $\wh f(x)$ now reads
\begin{align*}
    \wh f(x) = w_{\P} \wh f_{k_{\P}}(x) + w_{\Q} \wh f_{k_{\Q}}(x).
\end{align*}
Because the source and target samples are independent, the expression of $\ol f(x)$ reads similarly
\begin{align*}
    \ol f(x) = w_{\P} \ol f_{k_{\P}}(x) + w_{\Q} \ol f_{k_{\Q}}(x).
\end{align*}
Hence, for all $x \in \RR^d,$
\begin{align}
    \label{eq.two.sample.bv}
    (\wh f(x) - f(x))^2 &= \Big(w_{\P}(\wh f_{k_{\P}}(x) - f(x)) + w_{\Q}(\wh f_{k_{\Q}}(x) - f(x))\Big)^2\nonumber\\
    &\leq w_{\P}(\wh f_{k_{\P}}(x) - f(x))^2 + w_{\Q}(\wh f_{k_{\Q}}(x) - f(x))^2.
\end{align}
This leads to the risk being upper bounded as $\mE_{\Q_{\sX}}(\wh f, f) \leq w_{\P}\mE_{\Q_{\sX}}(\wh f_{k_{\P}}, f) + w_{\Q}\mE_{\Q_{\sX}}(\wh f_{k_{\Q}}, f),$ which allows for a straightforward proof of the following theorem.

\begin{theorem}
\label{th.transfer.two.samples.fixed.k.proof}
Let $\k_{\P}, \k_{\Q} >0, \tau > 1.$ Additionally, define
\begin{align*}
    u_T := \gamma\wedge \frac{2\beta}{d}, &\quad u_M := \frac{\rho}{\rho + d}\wedge \frac{2\beta}{d}\\
    r_T := \frac{\gamma}{\gamma + 1} \wedge \frac{2\beta}{2\beta + d}, &\quad r_M := \frac{\rho}{2\rho + d} \wedge \frac{2\beta}{2\beta + d}
\end{align*}
Let $\wh f$ be the two-sample $(k_{\P}, k_{\Q})$-NN regressor defined by \eqref{eq.two.sample.local.knn} with constant neighbour functions given by
\begin{align*}
    k_{\P} &= \lceil\k_{\P}\log(n + m)^{1/(r_T + 1)}n^{r_T/(r_T + 1)}\rceil,\qquad \text{and}\\
    k_{\Q} &= \lceil \k_{\Q}\log(n + m)^{1/(r_M + 1)}m^{r_M/(r_M + 1)}\rceil.
\end{align*}
Then, there exists $N = N(r_T, \k_{\P}) \geq 2d$ and $M = M(r_M, \k_{\Q}) \geq 2d$ such that for all $n \geq N$ and all $m\geq M,$ with probability at least $1 - 2n^{-\tau} - 2m^{-\tau}$ it holds that
\begin{align*}
    \mE_{\Q}(\wh f, f) \leq \frac{\k_{\P}C_{\P} + \k_{\Q}C_{\Q}}{\k_{\P} \wedge \k_{\Q}} \Bigg(\bigg(\frac{n}{\log(n + m)}\bigg)^{r_T/(r_T + 1)} + \bigg(\frac{m}{\log(n + m)}\bigg)^{r_M/(r_M + 1)}\Bigg)^{-1},\\
\end{align*}
where $C_{\P} = C_{\P}(\k_{\P}, r_T)$ and $C_{\Q} = C_{\Q}(\k_{\Q}, r_M)$ are the multiplicative constants of the rates of $f_{\P}$ and $f_{\Q}$ respectively, which are given by
\begin{align*}
    C_{R}(\k_{R}, t) := \Bigg[\frac{2C_\tau}{\k_R} + \bigg[2Lr_-^{2\beta} + 8F^2\bigg]\bigg(\frac{\k_R B_\tau}{a_-r_-^d}\bigg)^{t}\mT_{R_{\sX}}(t)\Bigg],
\end{align*}
for $R \in \{\P, \Q\}.$
\end{theorem}

\begin{proof}
Denote $\wh f  = w_{\P} \wh f_{\P} + w_{\Q}\wh f_{\Q}$ where $\wh f_{\P}$ and $\wh f_{\Q}$ are the $k_{\P}$ and $k_{\Q}$-NN estimators on the samples from $\P_{\sX, \sY}$ and $\Q_{\sX, \sY}$ respectively. A slight modification of the proof of Theorem \ref{th.rates.standard.knn.proof} shows that our choice of $k_{\P}$ and $k_{\Q}$ leads to
\begin{align*}
    \PP\Bigg\{\mE_{\Q}(\wh f_{\P}, f) &\leq C_{\P} \bigg(\frac{\log(n + m)}{n}\bigg)^{r_T/(r_T + 1)}\Bigg\} \geq 1 - 2n^{-\tau}\\
    \PP\Bigg\{\mE_{\Q}(\wh f_{\Q}, f) &\leq C_{\Q} \bigg(\frac{\log(n + m)}{m}\bigg)^{r_M/(r_M + 1)}\Bigg\} \geq 1 - 2m^{-\tau},
\end{align*}
while the constants $C_{\P}$ and $C_{\Q}$ remain the same as in the claim of Theorem \ref{th.rates.standard.knn.proof}. This enables us to work on the intersection of the previously considered events and derive,
\begin{align*}
    \mE_{\Q}(\wh f, f) &\leq \frac{k_{\P}}{k_{\P} + k_{\Q}}\mE_{\Q}(\wh f_{\P}, f) + \frac{k_{\Q}}{k_{\P} + k_{\Q}}\mE_{\Q}(\wh f_{\Q}, f)\\
    &\leq \frac{\k_{\P}C_{\P} + \k_{\Q}C_{\Q}}{\k_{\P} \wedge \k_{\Q}}\frac{\log(n + m)}{\log(n + m)^{(1/(r_T + 1)}n^{r_T/(r_T + 1)} + \log(n + m)^{1/(r_M + 1)}m^{r_M/(r_M + 1)}}\\
    &= \frac{\k_{\P}C_{\P} + \k_{\Q}C_{\Q}}{\k_{\P} \wedge \k_{\Q}}\Bigg(\bigg(\frac{n}{\log(n + m)}\bigg)^{r_T/(r_T + 1)} + \bigg(\frac{m}{\log(n + m)}\bigg)^{r_M/(r_M + 1)}\Bigg)^{-1}.\\
\end{align*}
Applying the union bound finishes the proof.
\end{proof}

To obtain our main result, we need an additional lemma that allows to lower bound $\zeta_{h_-}(x)$ for $x$ in a low-density region.
\begin{lemma}
    \label{lem.bounded.below.zeta.low.density}
    Let $\P \in \mP,$ and $k \in \{1, \dots, n\}.$ If $x \in \supp(\P)$ satisfies $p(x) \leq \alpha k/n,$ with $\alpha > 0,$ then, 
    \begin{align*}
         \zeta_{h_-}(x) \geq \bigg(\frac{D_\tau}{\alpha a_+}\bigg)^{1/d} \wedge r_+.
    \end{align*}
\end{lemma}

\begin{proof}
    If $\zeta_{h_-}(x) > r_+,$ then, the result is immediate. By contradiction, we show the claim in the case $\zeta_{h_-}(x) \leq r_+$. Let $t \in (0, r_+]$ be arbitrary and assume that $\zeta_{h_-}(x) \leq t^{1/d} \leq r_+.$ Then, by \eqref{eq.maximal.mass}, we have $h_- = D_\tau k/n \leq a_+tp(x).$ Rearranging the latter and using the assumption leads to
    \begin{align*}
        \frac{D_\tau k}{a_+ t n} \leq p(x) \leq \frac{\alpha k}n,
    \end{align*}
    which implies $t \geq D_\tau/(\alpha a_+).$ Hence
    \begin{align*}
         \zeta_{h_-}(x) \geq \bigg(\frac{D_\tau}{\alpha a_-}\bigg)^{1/d} \wedge r_+.
    \end{align*}
\end{proof}

\begin{theorem}[Rates for two-sample local nearest neighbours]
\label{th.rates.local.k.two.sample.proof}
    Let $\tau > 0, \k_{\P}, \k_{\Q} >0,$ and define
    \begin{align*}
        r_T := 2\beta/(2\beta + d) \vee \gamma,\text{ and }r_M := \rho/(\rho + d)\vee 2\beta/(2\beta + d).
    \end{align*}
    Let $\wh p$ and $\wh q$ be, respectively, the $\ell_{\P}$ and $\ell_{\Q}$-NN density estimators of $p$ and $q$ with
    \begin{align*}
        \ell_{\P} \geq \big\lceil V_\tau\log n\big\rceil, \text{ and }\ell_{\Q} &\geq \big\lceil V_\tau\log m\big\rceil.
    \end{align*}
    Finally, consider $\wh f$ to be the local $(k_{\P}, k_{\Q})$-NN estimator with neighbour functions given by
    \begin{align*}
        k_{\P}(x) &:= n\wedge \Big\lceil \kappa_{\P}\log(n + m)^{d/(2\beta + d)}(n\wh p(x))^{2\beta/(2\beta + d)}\Big\rceil \vee \big\lceil \log (n + m) \big\rceil\\
        k_{\Q}(x) &:= m\wedge \Big\lceil \kappa\log(n + m)^{d/(2\beta + d)}(m\wh q(x))^{2\beta/(2\beta + d)}\Big\rceil \vee \big\lceil \log (n + m) \big\rceil.
    \end{align*}
    Then, there exists $N \geq 2d$ and $M \geq 2d$ such that for all $n \geq N, m \geq M,$ it holds, with probability at least $1 - 3n^{1-\tau} - 3m^{1-\tau},$
    \begin{multline*}
        \mE_{\Q}(\wh f, f) \leq A_0\Bigg[\bigg(\frac{n}{\log(n+m)^{r_0/r_T}\ell_{\P}^{1 - r_0/r_T}}\bigg)^{r_T} + \bigg(\frac{m}{\log(n + m)^{r_0/r_M}\ell_{\Q}^{1 - r_0/r_M}}\bigg)^{r_M}\Bigg]^{-1}\\
        + 16F^2A_2\Bigg(\bigg(\frac{n}{\ell_{\P}}\bigg)^{r_T} + \bigg(\frac{m}{\ell_{\Q}}\bigg)^{r_M}\Bigg)^{-1} + 8F^2A_3 \Bigg(\bigg(\frac{n}{\ell_{\P}}\bigg)^{r_T} + \bigg(\frac m{\ell_{\Q}}\bigg)^{r_M}\bigg(\frac{\ell_{\Q}}{\ell_{\P}}\bigg)^{r_0}\Bigg)^{-1}\\
         + 8F^2A_4 \Bigg(\bigg(\frac{m}{\ell_{\Q}}\bigg)^{r_T} + \bigg(\frac n{\ell_{\P}}\bigg)^{r_M}\bigg(\frac{\ell_{\P}}{\ell_{\Q}}\bigg)^{r_0}\Bigg)^{-1},
    \end{multline*}
    with
    \begin{align*}
        r_0 &:= \frac{2\beta}{2\beta + d}\\
        \k &:= \k_{\P}^{(2\beta + d)/d} + \k_{\Q}^{(2\beta + d)/d}\\
        \alpha_p &:= (D_\tau/(C_pa_+))\wedge r_+^d\\
        \alpha_q &:= (D_\tau/(C_qa_+))\wedge r_+^d\\
        A_1 &:= \bigg[\frac{C_p^{r_T - r_0}\mT_{\P}}{\k_{\P}} \vee \frac{C_q^{r_M - r_0}\mT_{\Q}}{\k_{\Q}}\bigg]\cdot\bigg[4C_\tau\bigg(\frac{B_\tau}{a_-}\bigg)^{r_0} + 2\k L\bigg(\frac{a_+B_\tau}{a_-D_\tau}\bigg)^{2\beta/d + r_0}\bigg]\\
        A_2 &:= 2(\mT_{\P} \vee \mT_{\Q})(C_p^{r_T}\vee C_q^{r_M})\\
        A_3 &:= \big(C_p^{r_T}\mT_{\P}\big)\vee \frac{2\k_{\P}}{\k_{\Q}}\bigg(\frac{B_\tau}{\alpha_p C_q}\bigg)^{r_0}C_q^{r_M}\mT_{\Q}\\
        A_4 &:= \big(C_q^{r_M}\mT_{\Q}\big)\vee \frac{2\k_{\Q}}{\k_{\P}}\bigg(\frac{B_\tau}{\alpha_q C_p}\bigg)^{r_0}C_p^{r_T}\mT_{\P},
    \end{align*}
    and $C_p := C(\k_{\P}), \ C_q := C(\k_{\Q}),$ where
    \begin{align*}
        C(t) := c_u\vee c_\ell \vee \frac{B_\tau}{a_-t^{(2\beta + d)/(2\beta)}V_\tau} \vee \frac{(2c_ut)^{(2\beta + d)/d}}{V_\tau}\bigg(\frac{a_+}{B_\tau}\bigg)^{2\beta/d}.
    \end{align*}
    
\end{theorem}
\begin{proof}
    Let 
    \begin{align*}
        C(t) := c_u\vee c_\ell \vee \frac{B_\tau}{a_-t^{(2\beta + d)/(2\beta)}V_\tau} \vee \frac{(2c_ut)^{(2\beta + d)/d}}{V_\tau}\bigg(\frac{a_+}{B_\tau}\bigg)^{2\beta/d},
    \end{align*}
    and define $C_p := C(\k_{\P}), \ C_q := C(\k_{\Q}).$ For $R \in \{\P, \Q\},$ Define the event
    \begin{align*}
        E_n^{R_{\sX}}(\ell, \wh f_R) := L_n^{R_{\sX}} \cap \U_n^{R_{\sX}}(\ell) \cap V_n^{R_{\sX}}(\wh f_R).
    \end{align*}
    We work on the event
    \begin{align}
        \label{eq.event.two.sample.local}
        E_n^{\P_{\sX}}(\ell_{\P}, \wh f_{\P}) \cap E_m^{\Q_{\sX}}(\ell_{\Q}, \wh f_{\Q}).
    \end{align}
    Let $X \sim \Q_{\sX}$ independent of the samples. We start by decomposing the risk as in \eqref{eq.two.sample.bv}. We get
    \begin{align*}
        \mE_{\Q}(\wh f, f) &\leq \E\bigg[w_{\P}(X)\big(\wh f_{\P}(X) - f(X)\big)^2 + w_{\Q}(X) \big(\wh f_{\Q}(X) - f(X)\big)^2\bigg].
    \end{align*}
    Let $\Delta_{\P} := \Delta_n^{\P_{\sX}}(C_p, \ell_{\P})$ and $\Delta_{\Q} := \Delta_m^{\Q_{\sX}}(C_q, \ell_{\Q}).$ Consider also the set $\Delta := \Delta_{\P}\cap \Delta_{\Q}, \ \Gamma_{\P} := \Delta_{\P}\cap \Delta_{\Q}^c,\ \Gamma_{\Q} := \Delta_{\P}^c\cap \Delta_{\Q},$ and $\nabla := \Delta_{\P}^c\cap \Delta_{\Q}^c.$ As in Theorem \ref{th.transfer.local.knn}, we apply Proposition \ref{prop.decomp.knn} twice to obtain that on the event \eqref{eq.event.two.sample.local},
    \begin{align*}
        \mE_{\Q}(\wh f, f) &\leq \E\bigg[V(X) + b^2(X) + D(X)\bigg]
    \end{align*}
    where 
    \begin{align*}
        V(x) &:= 2C_\tau\Bigg[w_{\P}(x)\frac{\log n}{k_{\P}(x)} + w_{\Q}\frac{\log m}{k_{\Q}(x)}\Bigg] \leq \frac{4C_\tau\log(n + m)}{k_{\P}(x) + k_{\Q}(x)}\\
        b^2(x) &:= 2L\Bigg[w_{\P}(x) \bigg(\frac{B_\tau k_{\P}(x)}{a_-np(x)}\bigg)^{2\beta/d} + w_{\Q}(x) \bigg(\frac{B_\tau k_{\Q}(x)}{a_-np(x)}\bigg)^{2\beta/d}\Bigg]\mathds{1}(x \in \Delta),\\
        D(x) &:= 8F^2\Bigg[\frac{k_{\P}(x)}{k_{\P}(x) + k_{\Q}(x)}\mathds{1}(x \in \Gamma_{\Q}) + \frac{k_{\Q}(x)}{k_{\P}(x) + k_{\Q}(x)}\mathds{1}(x \in \Gamma_{\P})\Bigg]+ 16F^2\mathds{1}(x \in \nabla).
    \end{align*}
    We now denote $r_0 := 2\beta/(2\beta + d), \ \mT_{\P} := \int q(x)/p(x)^{r_T}\, dx$ and $\mT_{\Q} := \int q(x)^{1 - r_M}\, dx.$ We proceed as in \eqref{eq.local.one.sample.pointwise.variance.bound} (Definition of $k_{\P}, k_{\Q},$ Lemma \ref{lem.inclusions} and Lemma \ref{lem.density.ratio} $(ii)$) to obtain, for $x \in \Delta,$
    \begin{align}
    \label{eq.two.sample.local.variance.bound}
        V(x) &\leq 4C_\tau \Bigg(\k_{\P}\frac{\log(n+m)^{1 - r_0}(n\wh p(x))^{r_0}}{\log(n+ m)} + \k_{\Q}\frac{\log(n+m)^{1 - r_0}(m\wh q(x))^{r_0}}{\log(n+ m)}\Bigg)^{-1}\nonumber\\
        &\leq 4C_\tau \Bigg(\k_{\P}\bigg(\frac{a_-np(x)}{B_\tau\log(n+ m)}\bigg)^{r_0} + \k_{\Q}\bigg(\frac{a_-m q(x)}{B_\tau \log(n + m)}\bigg)^{r_0}\Bigg)^{-1}\nonumber\\
        &= 4C_\tau\bigg(\frac{B_\tau}{a_-}\bigg)^{r_0} \Bigg(\k_{\P}\bigg(\frac{np(x)}{\log(n+ m)}\bigg)^{r_0} + \k_{\Q}\bigg(\frac{m q(x)}{\log(n + m)}\bigg)^{r_0}\Bigg)^{-1}
    \end{align}
    Conversely, by definition of $k_{\P}, k_{\Q},$ for $x \in \Delta^c,$
    \begin{align*}
        V(x)\mathds{1}(x \in \Delta^c) &\leq 4C_\tau.
    \end{align*}
    Next, we proceed as in \eqref{eq.local.one.sample.pointwise.bias.hd.bound} and obtain, for $x \in \Delta,$
    \begin{align*}
        k_{\P}(x)\bigg(\frac{B_\tau k_{\P}(x)}{a_-np(x)}\bigg)^{2\beta/d} &= \bigg(\frac{B_\tau}{a_-}\bigg)^{2\beta/d}\frac{k_{\P}(x)^{(2\beta + d)/d}}{(np(x))^{2\beta/d}}\\
        &=\bigg(\frac{B_\tau}{a_-}\bigg)^{2\beta/d}\k_{\P}^{(2\beta + d)/d}\log(n+m)\bigg(\frac{\wh p(x)}{p(x)}\bigg)^{2\beta/d}\\
        &\leq \k_{\P}^{(2\beta + d)/d}\bigg(\frac{a_+B_\tau}{a_-D_\tau}\bigg)^{2\beta/d}\log(n+m).
    \end{align*}
    The same analysis for the second term in the expression of $b^2(x)$ yields
    \begin{align*}
        k_{\Q}(x)\bigg(\frac{B_\tau k_{\Q}(x)}{a_-mq(x)}\bigg)^{2\beta/d} \leq \k_{\Q}^{(2\beta + d)/d}\bigg(\frac{a_+B_\tau}{a_-D_\tau}\bigg)^{2\beta/d}\log(n+m).
    \end{align*}
    Hence,
    \begin{align*}
        b^2(x) &\leq 2\k L\bigg(\frac{a_+B_\tau}{a_-D_\tau}\bigg)^{2\beta/d}\frac{\log(n+m)}{k_{\P}(x) + k_{\Q}(x)},
    \end{align*}
    where $\k := \k_{\P}^{(2\beta + d)/d} + \k_{\Q}^{(2\beta + d)/d}.$ The same analysis as in \eqref{eq.two.sample.local.variance.bound} leads to
    \begin{align}
        \label{eq.two.sample.local.bias.bound}
        b^2(x) &\leq 2\k L\bigg(\frac{a_+B_\tau}{a_-D_\tau}\bigg)^{2\beta/d}\Bigg(\bigg(\frac{a_-np(x)}{B_\tau\log(n+ m)}\bigg)^{r_0} + \bigg(\frac{a_-m q(x)}{B_\tau \log(n + m)}\bigg)^{r_0}\Bigg)^{-1}\nonumber\\
        &= 2\k L\bigg(\frac{a_+B_\tau}{a_-D_\tau}\bigg)^{2\beta/d + r_0}\Bigg(\k_{\P}\bigg(\frac{np(x)}{\log(n+ m)}\bigg)^{r_0} + \k_{\Q}\bigg(\frac{m q(x)}{\log(n + m)}\bigg)^{r_0}\Bigg)^{-1}.
    \end{align}
    We now integrate the pointwise bounds derived above with respect to $\Q_{\sX}.$ We have
    \begin{align*}
        \int_{\Delta} &\Bigg(\k_{\P}\bigg(\frac{np(x)}{\log(n+ m)}\bigg)^{r_0} + \k_{\Q}\bigg(\frac{q(x)}{\log(n + m)}\bigg)^{r_0}\Bigg)^{-1}\, d\Q_{\sX}(x)\\
        &\leq \int_{\Delta} \Bigg(\k_{\P}\bigg(\frac{np(x)}{\log(n+ m)}\bigg)^{r_0} \vee \k_{\Q}\bigg(\frac{m q(x)}{\log(n + m)}\bigg)^{r_0}\Bigg)^{-1}\, d\Q_{\sX}(x)\\
        &\leq \frac 1{\k_{\P}}\int_{\Delta} \bigg(\frac{\log(n+ m)}{np(x)}\bigg)^{r}\, d\Q_{\sX}(x) \wedge \frac 1{\k_{\Q}}\int_{\Delta}\bigg(\frac{\log(n + m)}{m q(x)}\bigg)^{r_0}\, d\Q_{\sX}(x)\\
        &\leq \log(n+m)^{r_0}\Bigg[\frac 1{\k_{\P}n^{r_0}}\int_{\Delta} \frac{q(x)}{p(x)^{r_T}}p(x)^{r_T - r_0}\, dx \wedge \frac{1}{\k_{\Q}m^{r_0}}\int_{\Delta} \frac{q(x)}{q(x)^{r_M}}q(x)^{r_M - r_0}\, dx\Bigg]\\
        &\leq \log(n+m)^{r_0}\Bigg[\frac{1}{\k_{\P}n^{r_0}}\bigg(\frac{C_p\ell_{\P}}{n}\bigg)^{r_T - r_0}\mT_{\P} \wedge \frac{1}{\k_{\Q}m^{r_0}}\bigg(\frac{C_q\ell_{\Q}}{m}\bigg)^{r_M - r_0}\mT_{\Q}\Bigg]\\
        &\leq A_0\Bigg[\bigg(\frac{n}{\log(n+m)^{r_0/r_T}\ell_{\P}^{1 - r_0/r_T}}\bigg)^{r_T} + \bigg(\frac{m}{\log(n + m)^{r_0/r_M}\ell_{\Q}^{1 - r_0/r_M}}\bigg)^{r_M}\Bigg]^{-1},
    \end{align*}
    where $A_0$ is defined as
    \begin{align*}
        A_0 := \frac{C_p^{r_T - r_0}\mT_{\P}}{\k_{\P}} \vee \frac{C_q^{r_M - r_0}\mT_{\Q}}{\k_{\Q}}.
    \end{align*}
    Ultimately, we bound the value of $\E[D(X)].$ We proceed as in \eqref{eq.local.one.sample.bias.lp.bound}. $\E[D(X)]$ is comprised of three terms. The third one is $16F^2\Q_{\sX}\big(\nabla\big) = 16F^2\int_{\nabla} q(x)\, dx.$ The density ratio exponent assumption implies that $\lambda(\supp(\Q_{\sX}) \setminus \supp(\P_{\sX})) = 0,$ where $\lambda$ is the Lebesgue measure on $\RR^d.$ This enables the derivation
    \begin{align*}
        \int_{\nabla} q(x)\, dx &= \int_{\nabla} \frac{q(x)}{p(x)^{r_T}\wedge q(x)^{r_M}} \big(p(x)^{r_T} \wedge q(x)^{r_M}\big)\, dx\\
        &\leq \Bigg(\bigg(\frac{C_p\ell_{\P}}{n}\bigg)^{r_T} \wedge \bigg(\frac{C_q\ell_{\Q}}{m}\bigg)^{r_M}\Bigg)\int_{\nabla} \frac{q(x)}{p(x)^{r_T}} \vee \frac{q(x)}{q(x)^{r_M}}\, dx\\
        &\leq \Bigg(\bigg(\frac{C_p\ell_{\P}}{n}\bigg)^{r_T} \wedge \bigg(\frac{C_q\ell_{\Q}}{m}\bigg)^{r_M}\Bigg)\int_{\nabla} \frac{q(x)}{p(x)^{r_T}} + \frac{q(x)}{q(x)^{r_M}}\, dx\\
        &\leq \Bigg(\bigg(\frac{C_p\ell_{\P}}{n}\bigg)^{r_T} \wedge \bigg(\frac{C_q\ell_{\Q}}{m}\bigg)^{r_M}\Bigg)\big(\mT_{\P} + \mT_{\Q}\big)\\
        &\leq A_2\Bigg(\bigg(\frac{n}{\ell_{\P}}\bigg)^{r_T} + \bigg(\frac{m}{\ell_{\Q}}\bigg)^{r_M}\Bigg)^{-1},
    \end{align*}
    with 
    \begin{align*}
        A_2 := 2(\mT_{\P} \vee \mT_{\Q})(C_p^{r_T}\vee C_q^{r_M}).
    \end{align*}
    Next, we bound the two other terms in the expression of $\E[D(X)].$ We start with $8F^2\int_{\Gamma_{\Q}}k_{\P}(x)/(k_{\P}(x) + k_{\Q}(x))\, d\Q_{\sX}(x).$ On $\Gamma_{\Q},$
    \begin{align*}
        \frac{k_{\P}(x)}{k_{\P}(x) + k_{\Q}(x)} &\leq 1 \wedge \frac{k_{\P}(x)}{k_{\Q}(x)},
    \end{align*}
    and, by Lemma \ref{lem.inclusions} $(i),$ it holds that for $n \geq N, \ k(x) \leq 2\k_{\P}\log(n + m)^{1 - r_0}(n\wh p(x))^{r_0},$ therefore,
    \begin{align*}
        \frac{k_{\P}(x)}{k_{\Q}(x)} &\leq \frac{2\k_{\P}}{\k_{\Q}}\bigg(\frac{\ell_{\P}}{\ell_{\Q}}\frac{R^{\Q}_{\ell_{\Q}}(x)^d}{R^{\P}_{\ell_{\P}}(x)^d}\bigg)^{r_0}
    \end{align*}
    Applying Lemma \ref{lem.lower.bound.neighbours.density}, Lemma \ref{lem.bounded.below.zeta.low.density} with $\alpha = C_p,$ and $x \in \Delta_{\P}^c,$ we can bound $R_{\ell_{\P}}^{\P}(x)^d \geq (D_\tau/(C_pa_+))\wedge r_+^d =: \alpha_p.$ Additionally, we bound $R_{\ell_{\Q}}^{\Q}(x)^d$ using Lemma \ref{lem.bound.neighbours.density} and we obtain
    \begin{align*}
        \bigg(\frac{\ell_{\P}}{\ell_{\Q}}\frac{R^{\Q}_{\ell_{\Q}}(x)^d}{R^{\P}_{\ell_{\P}}(x)^d}\bigg)^{r_0} &\leq \bigg(\frac{B_\tau}{\alpha_pa_-}\frac{\ell_{\P}}{ mq(x)}\bigg)^{r_0}.
    \end{align*}
    Hence,
    \begin{align*}
        \int_{\Gamma_{\Q}}\frac{k_{\P}(x)}{k_{\P}(x) + k_{\Q}(x)} q(x)\, dx &\leq \int_{\Gamma_{\Q}} q(x) \wedge \Bigg[\frac{2\k_{\P}}{\k_{\Q}}\bigg(\frac{B_\tau\ell_{\P}}{\alpha_pa_-mq(x)}\bigg)^{r_0}q(x)\Bigg]\, dx\\
        &\leq \Bigg[\int_{\Gamma_{\Q}} q(x)\, dx\Bigg] \wedge \int_{\Gamma_{\Q}}\frac{2\k_{\P}}{\k_{\Q}}\bigg(\frac{B_\tau\ell_{\P}}{\alpha_pa_-mq(x)}\bigg)^{r_0}q(x)\, dx\\
        &\leq \Bigg[\bigg(\frac{C_p \ell_{\P}}{n}\bigg)^{r_T}\mT_{\P}\Bigg] \wedge \Bigg[\frac{2\k_{\P}}{\k_{\Q}}\bigg(\frac{B_\tau \ell_{\P}}{\alpha_pa_-m}\bigg)^{r_0}\bigg(\frac{C_q \ell_{\Q}}{m}\bigg)^{r_M - r_0}\mT_{\Q}\Bigg].
    \end{align*}
    Let $\alpha_q := (D_\tau/(C_qa_+))\wedge r_+^d.$ A similar analysis leads to
    \begin{align*}
        \int_{\Gamma_{\P}} \frac{k_{\P}(x)}{k_{\P}(x) + k_{\Q}(x)}q(x)\, dx &\leq \Bigg[\bigg(\frac{C_q \ell_{\Q}}{m}\bigg)^{r_M}\mT_{\Q}\Bigg] \wedge \Bigg[\frac{2\k_{\Q}}{\k_{\P}}\bigg(\frac{B_\tau \ell_{\Q}}{\alpha_qa_-n}\bigg)^{r_0}\bigg(\frac{C_p \ell_{\P}}{n}\bigg)^{r_T - r_0}\mT_{\P}\Bigg].
    \end{align*}
    Finally,
    \begin{align*}
        \int_{\Gamma_{\Q}}\frac{k_{\P}(x)}{k_{\P}(x) + k_{\Q}(x)}q(x)\, dx &\leq A_3 \Bigg(\bigg(\frac{n}{\ell_{\P}}\bigg)^{r_T} + \bigg(\frac m{\ell_{\Q}}\bigg)^{r_M}\bigg(\frac{\ell_{\Q}}{\ell_{\P}}\bigg)^{r_0}\Bigg)^{-1}\\
        \int_{\Gamma_{\P}}\frac{k_{\Q}(x)}{k_{\P}(x) + k_{\Q}(x)}q(x)\, dx &\leq A_4 \Bigg(\bigg(\frac{m}{\ell_{\Q}}\bigg)^{r_T} + \bigg(\frac n{\ell_{\P}}\bigg)^{r_M}\bigg(\frac{\ell_{\P}}{\ell_{\Q}}\bigg)^{r_0}\Bigg)^{-1},
    \end{align*}
    with
    \begin{align*}
        A_3 &:= \big(C_p^{r_T}\mT_{\P}\big)\vee \frac{2\k_{\P}}{\k_{\Q}}\bigg(\frac{B_\tau}{\alpha_p C_q}\bigg)^{r_0}C_q^{r_M}\mT_{\Q}\\
        A_4 &:= \big(C_q^{r_M}\mT_{\Q}\big)\vee \frac{2\k_{\Q}}{\k_{\P}}\bigg(\frac{B_\tau}{\alpha_q C_p}\bigg)^{r_0}C_p^{r_T}\mT_{\P}.\\
    \end{align*}
    Combining everything yields
    \begin{multline*}
        \mE_{\Q}(\wh f, f) \leq A_1\Bigg[\bigg(\frac{n}{\log(n+m)^{r_0/r_T}\ell_{\P}^{1 - r_0/r_T}}\bigg)^{r_T} + \bigg(\frac{m}{\log(n + m)^{r_0/r_M}\ell_{\Q}^{1 - r_0/r_M}}\bigg)^{r_M}\Bigg]^{-1}\\
        + 16F^2A_2\Bigg(\bigg(\frac{n}{\ell_{\P}}\bigg)^{r_T} + \bigg(\frac{m}{\ell_{\Q}}\bigg)^{r_M}\Bigg)^{-1} + 8F^2A_3 \Bigg(\bigg(\frac{n}{\ell_{\P}}\bigg)^{r_T} + \bigg(\frac m{\ell_{\Q}}\bigg)^{r_M}\bigg(\frac{\ell_{\Q}}{\ell_{\P}}\bigg)^{r_0}\Bigg)^{-1}\\
         + 8F^2A_4 \Bigg(\bigg(\frac{m}{\ell_{\Q}}\bigg)^{r_T} + \bigg(\frac n{\ell_{\P}}\bigg)^{r_M}\bigg(\frac{\ell_{\P}}{\ell_{\Q}}\bigg)^{r_0}\Bigg)^{-1},
    \end{multline*}
    with
    \begin{align*}
        A_1 := \bigg[4C_\tau\bigg(\frac{B_\tau}{a_-}\bigg)^{r_0} + 2\k L\bigg(\frac{a_+B_\tau}{a_-D_\tau}\bigg)^{2\beta/d + r_0}\bigg]A_0.
    \end{align*}
    We use Proposition \ref{prop.neighbours.distance.bound.zeta}, Proposition \ref{prop.neighbours.distance.lower.bound.zeta}, Corollary \ref{cor.variance.bound.nice.version}, and Remark \ref{rem.bound.variable.k} together with the union bound to obtain
    \begin{align*}
        \PP\bigg\{E_n^{\P_{\sX}}(\ell_{\P}, \wh f_{\P}) \cap E_m^{\Q_{\sX}}(\ell_{\Q}, \wh f_{\Q})\bigg\} \geq 1 - 2n^{-\tau} - n^{1-\tau} - 2m^{-\tau} - m^{1-\tau} \geq 1 - 3n^{1-\tau} - 3m^{1-\tau}.
    \end{align*}
\end{proof}

\begin{remark}
    The result of Theorem \ref{th.transfer.local.knn} $(ii)$ is obtained by taking $\ell_{\P} = \ell_{\Q} = \lceil V_\tau\log(n + m)\rceil$ in the regime $n\wedge m \geq \lceil V_\tau\log (n + m)\rceil.$
\end{remark}

\begin{lemma}
    \label{lem.example.exponential}
    Let $\P_{\sX}  = \mE(\lambda),$ then, $\P_{\sX} \in \mP(\lambda, \exp(-\lambda), 1, 2\sinh(\lambda)/\lambda, 1)$ and $\P_{\sX}$ satisfies \eqref{eq.pseudo.moment} for any $\rho \geq 0.$ Additionally, if $\P_{\sX} = \mE(\lambda_{\P}), \ \Q_{\sX} = \mE(\lambda_{\Q}),$ and $\lambda = \lambda_{\P} \vee \lambda_{\Q},$ then, $\P_{\sX}$ and $\Q_{\sX}$ are both in $\mP(\lambda, \exp(-\lambda), 1, 2\sinh(\lambda)/\lambda, 1),$ and $\Q_{\sX} \in \mQ(\P_{\sX}, \gamma, \rho)$ for all $\gamma < \lambda_{\Q}/\lambda_{\P}$ and all $\rho \geq 0.$
\end{lemma}

\begin{proof}
    We first check the assumptions in Definition \ref{def.class.proba}. By definition, the density of $\P_{\sX}$ is upper bounded by $\lambda.$ Next, for any $x > 0$ and $0 < t \leq 1,$
    \begin{align*}
        \P_{\sX}\{\B(x, t)\} = \int_{x - t}^{x + t} \lambda\exp(-\lambda u)\mathds{1}(u \geq 0)\, du &\geq \int_x^{x + t}\lambda\exp(-\lambda (x + t))\, du\\
        &= t\exp(-\lambda t)p(x) \geq a_-tp(x),
    \end{align*}
    where $a_- := \exp(-\lambda).$ This proves that the minimal mass property is satisfied for any $0 < t \leq 1$ and $a_-$ defined previously. Next, we prove the maximal mass property. Let $t \in (0, 1].$ If $0 \leq x \leq t,$ then $\P_{\sX}\{\B(x, t)\} = 1 - \exp(-\lambda(x + t)),$ and
    \begin{align*}
        \frac{\P_{\sX}\{\B(x, t)\}}{tp(x)} = \frac{1 - \exp(-\lambda(x + t))}{t\lambda\exp(-\lambda x)} &= \frac{\exp(\lambda x) - \exp(-\lambda t)}{t\lambda}\\
        &\leq \frac{\exp(\lambda t) - \exp(-\lambda t)}{t\lambda} = \frac{2\sinh(\lambda t)}{\lambda t}.
    \end{align*}
    The function $g \colon t \mapsto 2\sinh(\lambda t)/(\lambda t)$ continuously increases on $\RR_+^*.$ Moreover, $g$ admits a limit as $t \to 0,$ which is equal to $2.$ Let then $a_+ := g(1) = \sinh(\lambda)/\lambda,$ we have proven that for $0 < t \leq 1$ and $0\leq x \leq t,$ it holds that $\P_{\sX}\{\B(x, t)\} \leq a_+p(x)t.$ We now consider $t < x,$ in which case we have
    \begin{align*}
        \frac{\P_{\sX}\{\B(x, t)\}}{tp(x)} = \frac{\exp(-\lambda(x - t)) - \exp(-\lambda(x + t))}{\lambda t \exp(-\lambda x)} = \frac{2\sinh(\lambda t)}{\lambda t} \leq a_+,
    \end{align*}
    as previously, for all $0 < t \leq 1.$ We have shown that $\P_{\sX} \in \mP(\lambda, \exp(-\lambda), 1, 2\sinh(\lambda)/\lambda, 1).$ Next, we prove the pseudo moment assumption \eqref{eq.pseudo.moment}. Let $\rho > 0$ arbitrary. We have
    \begin{align*}
        \int p(x)^{d/(\rho + d)}\, dx = \int \lambda^{d/(\rho + d)} \exp\bigg(-\frac{\lambda xd}{\rho + d}\bigg)\, dx < \infty.
    \end{align*}
    Hence, $\P_{\sX}$ satisfies \eqref{eq.pseudo.moment} for all $\rho > 0.$ Finally, we show the second claim. Let $\gamma < \lambda_{\Q}/\lambda_{\P}.$ Then,
    \begin{align*}
        \int \frac{q(x)}{p(x)^{\gamma}}\, dx &= \frac{\lambda_{\Q}}{\lambda_{\P}^{\gamma}} \int \exp((\gamma\lambda_{\P} - \lambda_{\Q})x)\, dx < \infty,
    \end{align*}
    which finishes the proof.
\end{proof}

\begin{lemma}
    \label{lem.pareto}
    Let $\P_{\sX}$ and $\Q_{\sX}$ be two Pareto distributions with parameters $\alpha_{\P}$ and $\alpha_{\Q}$ respectively, that is, $p(x) = x^{-(\alpha_{\P} + 1)}\mathds{1}(x \geq 1)$ and $q(x) = x^{-(\alpha_{\Q} + 1)}\mathds{1}(x \geq 1).$ Then, the distribution $\P_{\sX}$ belongs to $\mP(1, 2, 1/2, a(\alpha_{\P}), 1/2)$ with $a\colon \alpha \mapsto 3^{\alpha+1}/2^{\alpha + 2}.$ Additionally, $\Q_{\sX} \in \mP(\P_{\sX}, \gamma, \rho)$ for all $\gamma < \alpha_{\Q}/(1 + \alpha_{\P})$ and all $\rho > \alpha_{\Q}.$ Consequently, $(\P_{\sX}, \Q_{\sX}) \in \mP(1,2,1/2,a(\alpha_{\P}\vee\alpha_{\Q}), 1/2).$
\end{lemma}

\begin{proof}
    Let $p$ be the density of $\P_{\sX}$ and $r_+ = r_- = 1/2.$ We have, for all $ 0 < r \leq 1/2,$
    \begin{align*}
        \frac{2(p(x - r)\wedge 1)}{p(x)} &= x^{\alpha_{\P} + 1}\mathds{1}(x \leq 1 + r) + \bigg(\frac{x}{x - r}\bigg)^{\alpha_{\P} + 1}\mathds{1}(x > 1 +r)\\
        &\leq (1 + r)^{\alpha_{\P} + 1} \leq (1 + r_+)^{\alpha_{\P} + 1} =  \bigg(\frac{3}{2}\bigg)^{\alpha_{\P} + 1} = a(\alpha_{\P}).
    \end{align*}
    This implies the maximal mass property since
    \begin{align*}
        \P_{\sX}\{\B(x, r)\} = \int_{(x - r)\vee 1}^{x + r}p(x)\, dx \leq 2rp((x - r)\wedge 1) \leq a(\alpha_{\P})p(x) r.
    \end{align*}
    On the other hand, we have, for all $0 < r \leq 1/2,$
    \begin{align*}
        2\frac{p(x + r)}{p(x)} &= \bigg(\frac{x + r}{x}\bigg)^{\alpha_{\P} + 1} \geq 2,
    \end{align*}
    which shows that the minimal mass property holds since,
    \begin{align*}
        \P_{\sX}\{\B(x, r)\} = \int_{(x - r)\vee 1}^{x + r}p(x)\, dx \geq rp(x + r) \geq 2rp(x).
    \end{align*}
    We now show the second claim. We can check that
    \begin{align*}
        \int x^{-(\alpha_{\Q} + 1)/(\rho + 1)}\mathds{1}(x \geq 1)\, dx
    \end{align*}
    converges if and only if $\rho > \alpha_{\Q}.$ Finally,
    \begin{align*}
        \int \frac{q(x)}{p(x)^\gamma}\, dx = \int x^{\gamma(\alpha_{\P} + 1) - \alpha_{\Q} - 1}\mathds{1}(x \geq 1)\, dx
    \end{align*}
    converges if and only if $\gamma < \alpha_{\Q}/(\alpha_{\P} + 1).$ This finishes the proof.
\end{proof}

\printbibliography

\end{document}